\documentclass{amsart}

\usepackage[latin1]{inputenc}
\usepackage{enumerate}
\usepackage{hyperref}
\usepackage{amsfonts,amsmath,amssymb}

\usepackage[ruled,vlined]{algorithm2e}
\def\return{\texttt{return }}

\usepackage{bbding} 

\usepackage{pgf}
\usepackage{tikz}
\usepgfmodule{plot}
\usepgflibrary{plothandlers}
\usetikzlibrary{shapes.geometric}

\newtheorem{theorem}{Theorem}
\newtheorem{lemma}[theorem]{Lemma}
\newtheorem{corollary}[theorem]{Corollary}
\newtheorem{proposition}[theorem]{Proposition}

\newtheorem{defn}[theorem]{Definition}

\theoremstyle{remark}
\newtheorem{example}[theorem]{Example}
\newtheorem{remark}[theorem]{Remark}

\def\N{{\mathbb N}}

\def\Z{{\mathbb Z}}

\def\pe[#1]{{\left\lceil #1\right\rceil}}
\def\peb[#1]{{\left\lfloor #1\right\rfloor}}
 
\def\GAP{\textsf{GAP}}
\title[Feng-Rao numbers of semigroups generated by intervals]{On the generalized Feng-Rao numbers of numerical semigroups generated by intervals}
\author{M. Delgado}
\address{CMUP, Departamento de Matematica, Faculdade de Ciencias, Universidade do Porto, Rua do Campo Alegre 687, 4169-007 Porto, Portugal}
\email{mdelgado@fc.up.pt}
\thanks{The first author acknowledges the support of the Centro de Matem\'atica da Universidade do Porto financed by FCT through the programs POCTI and POSI, with Portuguese and European Community structural funds, as well as the support of the FCT project PTDC/MAT/65481/2006.}

\author{J. I. Farr\'{a}n}
\address{Departamento de Matem\'atica Aplicada, Escuela Universitaria de Inform\'atica, 
Campus de Segovia - Universidad de Valladolid, Plaza de Santa Eulalia 9 y 11 - 40005 Segovia, Spain}
\email{jifarran@eii.uva.es}
\thanks{The second author is supported by the project MICINN-MTM-2007-64704.}

\author{P. A. Garc\'{\i}a-S\'{a}nchez}
\address{Departamento de \'Algebra, Universidad de Granada, 18071 Granada, Espa\~na}
\email{pedro@ugr.es}
 
\author{D. Llena}
\address{Departamento de Geometr\'{\i}a, Topolog\'{\i}a y Qu\'{\i}mica Org\'anica, Universidad de Almer\'{\i}a, 04120 Almer\'{\i}a, Espa\~na}
\email{dllena@ual.es}

\thanks{The third and fourth authors are supported by the projects MTM2010-15595, FQM-343 and FEDER funds}

\thanks{The third author is also supported by the project FQM-5849.}
\date{\today}



\begin{document}

\begin{abstract}
We give some general results concerning the computation of the generalized Feng-Rao numbers of numerical semigroups. In the case of a numerical semigroup generated by an interval, a formula for the $r^{th}$ Feng-Rao number is obtained.
\end{abstract}

\maketitle


\section{Introduction}

The Feng-Rao distance for a numerical semigroup was introduced in coding theory as 
a lower bound for the minimum distance of a one-point algebraic geometry (error-correcting) code (see \cite{FR}). 
This {\em order bound}\/, computed from Weierstrass semigroups, 
improves the lower bound for the minimum distance given by Goppa with the aid of the Riemann-Roch theorem. 
Moreover, the Feng-Rao distance is essential in a majority voting decoding 
procedure, that is the most efficient one for such kind of codes (see \cite{HvLP}). 

Even though the Feng-Rao distance was introduced for Weierstrass semigroups and 
for decoding purposes, it is just a combinatorial concept that makes sense for 
arbitrary numerical semigroups. This problem has been broadly studied in the 
literature for different types of semigroups (see \cite{WSPink}, \cite{Arf} or \cite{KirPel}). 
In numerical terms, the above mentioned improvement of the Goppa distance in coding theory means the following: 
For a semigroup $S$ with genus $g$ and $m\in S$ the {\em Feng-Rao distance}\/ satisfies 
\[
\delta_{FR}(m+1)\geq m+2-2g
\]
if $m>2g-2$, and equality holds for $m>>0$. 

On the other hand, the concept of minimum distance for an error-correcting code 
has been generalized to the so-called {\em generalized Hamming weights}\/. 
They were introduced independently by Helleseth et al. in \cite{HKM} and 
Wei in \cite{Wei} for applications in coding theory and cryptography, respectively. 

The natural generalization of the Feng-Rao distance to higher weights was introduced 
in \cite{HeiPel}. The computation of these generalized Feng-Rao distances turns out 
to be a very hard problem. Actually, very few results are known about this subject, 
and they are completely scattered in the literature 
(see for example \cite{AngCar}, \cite{HeiPel} or \cite{WCC}). 

This paper studies the asymptotical behaviour of the generalized Feng-Rao 
distances, that is, $\delta_{FR}^{r}(m)$ for $r\geq 2$ and $m>>0$. 
In fact, it was proven in \cite{WCC} that 
\begin{equation}\label{Er}
\delta_{FR}^{r}(m)=m+1-2g+E_{r}
\end{equation}
for $m>>0$ (details in the next section). The number $E_{r}\equiv \mathrm E(S,r)$ 
is called the $r$-th Feng-Rao number of the semigroup $S$, and they are 
unknown but for very few semigroups and concrete $r$'s. 
For example, it was proven in \cite{JMDA} that 
\[
\textrm E(S,r)=\rho_{r}
\]
for hyper-elliptic semigroups $S=\langle 2,2g+1\rangle$\/, with multiplicity 2 and genus $g$\/, 
and for Hermitian-like semigroups $S=\langle a,a+1\rangle$\/, 
where $S=\{\rho_1=0<\rho_2<\cdots\}$\/. In fact, it is not even known yet if 
this formula holds for arbitrary numerical semigroups generated by two elements $S=\langle a,b\rangle$\/. 
Nevertheless, our experimental results point in this direction. 

The main purpose of this paper is precisely to compute $\mathrm E(S,r)$ 
for semigroups generated by intervals, as a certain generalization of the Hermitian-like case. 
As a byproduct, we provide some general algorithms, implemented in GAP \cite{numericalsgps}, 
to compute Feng-Rao numbers.

The paper is written as follows. Section 2 presents the general definitions concerning numerical semigroups, 
Feng-Rao distances and Feng-Rao numbers, and some convenient visualizations of integers for a given semigroup. 
The reader may find useful to see some images in Subsection~\ref{convenient.visualisation}. 

The concept of amenable subset of a numerical semigroup is introduced in section 3. It consists of 
a set that is closed for taking divisors. It implies that distances between elements are somehow controlled. 
Amenable sets play a fundamental role in some general results on Feng-Rao numbers of numerical semigroups. 
These results allowed the implementation of a function to compute the Feng-Rao numbers of a numerical semigroup 
which works quite well. It uses some of the functionalities of the \GAP\ package \textsf{numericalsgps}~\cite{numericalsgps} 
and will hopefully be part of a future release of that package. We give in this way some general results. 
Among them, an important lemma shows that the divisors of a configuration 
are the divisors of the shadow plus the elements above the ground.

Many examples computed with the referred function helped us to gain the necessary intuition to obtain a formula for the 
$r^{th}$ Feng-Rao number of a numerical semigroup generated by an interval, which is presented in Section 4. 
This is the last and main result of this paper, and we briefly explain it in the sequel. 

Recall that we are aiming to find a formula for $\delta^r(m)$, when $S$ is a semigroup generated by an interval of integers. 
The strategy will be as follows: Suppose that there is an amenable set $M$ which is an optimal configuration whose shadow 
$L_M=[m,m+a+b)\cap M$ does not contain the ground (that is, $L_M\ne [m,m+a+b)\cap\N$). 
Then, using the results of Subsection~\ref{subsec:ordered_amenable}, we can construct an $r$-amenable set $N$ 
(said to be ordered amenable) whose shadow $L_N$ is an interval starting in $m$ and has no more elements than $\sharp L_M$. 
Furthermore, by Lemma~\ref{seguidos-dan-menos}, $\sharp\mathrm D(L_N)\le\sharp\mathrm D(L_M)$ which implies that the number 
of divisors of $N$ is no bigger than the number of divisors of $M$ and therefore $N$ is also an optimal configuration. 
It follows that ordered amenable sets are optimal configurations. 
Thus, the problem of computing the generalized Feng-Rao numbers is reduced to counting the divisors of intervals 
of the form $[m,m+\ell]\cap \N$, with $\ell\le a+b-1$. This is done by Corollary~\ref{cor:divs_intervals}. 
The main result, which gives a formula, then follows. 

\section{Definitions and basic results}
This section is divided into several subsections. We start with several basic definitions and we introduce some notation. The reader is referred to the book~\cite{NS} for details. Then we give the definition of generalized Feng-Rao numbers and end the section by giving a way to visualize the integers which is convenient for our purposes.
\subsection{Basic definitions and notation}

Let $S$ be a numerical semigroup, that is, a submonoid of $\N$
such that $\sharp(\N\setminus S)<\infty$ and $0\in S$\/.
Denote respectively by $g:=\sharp(\N\setminus S)$ and $c\in S$
the {\em genus} and the {\em conductor} of $S$\/, being $c$
by definition the (unique) element in $S$ such that $c-1\notin S$
and $c+l\in S$ for all $l\in\N$. Note that if $S$ is the Weierstrass
semigroup of a curve $\chi$ at a point $P$\/, $g$ equals to the
geometric genus of $\chi$\/, and the elements of $G(S):=\N\setminus S$
are called the Weierstrass gaps at $P$\/. For an arbitrary semigroup, these elements are simply called gaps.

It is well known (see for instance \cite[Lemma 2.14]{NS}) that $c\leq 2g$\/, and thus the \lq\lq largest gap\rq\rq \ of $S$
is $c-1\leq 2g-1$. The number $c-1$ is precisely the {\em Frobenius number}
of $S$\/. 
The \emph{multiplicity} of a numerical semigroup is the least positive integer belonging to it.

We say that a numerical semigroup $S$ is generated by a set of elements
$G\subseteq S$ if every element $x\in S$ can be written as a linear combination
\[
x=\displaystyle\sum_{g\in G}\lambda_{g}g,
\]
where finitely many $\lambda_{g}\in\N$ are non-zero.
In fact, it is classically known that every numerical semigroup
is finitely generated, that is, we can find a finite set $G$ generating $S$\/.
Furthermore, every generator set contains the set of irreducible elements,
$x\in S$ being irreducible if $x=u+v$ and $u,v\in S$ implies $u\cdot v=0$,
and this set actually generates $S$, so that it is usually called \lq\lq the\rq\rq \
generator set of $S$, whose cardinality is called {\em embedding dimension} of $S$
(more details in \cite{NS}).
Most of the times, we will suppose $S$ is minimally generated by $\{n_1<\cdots<n_e\}$. Its embedding dimension is $e$. Note that if $a$ and $b$ are integers, with $b<a$, and $S$ is minimally generated by the interval $[a,a+b]\cap \N$, then $n_1$ is $a$, $n_e-n_1$ is $b$ and the embedding dimension is $b+1$.  

Finally, if we enumerate the elements of $S$ in increasing order
\[
S=\{\rho_1=0<\rho_2<\cdots\},
\]
we note that every $x\geq c$ is the $(x+1-g)$\/-th element of $S$\/,
that is $x=\rho_{x+1-g}\,$.

\smallskip

The last part of this paper will be devoted to semigroups generated by intervals.

Let $a$ be a positive integer and $b$ an integer with $0<b<a$. Let $S=\langle a,a+1,\ldots,a+b\rangle$. Then $S$ is a numerical semigroup with multiplicity $a$ and embedding dimension $b+1$. As usual, let $c$ denote the conductor of $S$ and $m\ge 2c-1$. 

\subsection{Feng-Rao numbers}

Next we introduce the definitions for generalized Feng-Rao distances.
Although there is a subsection dedicated to the concept of divisor, we already need the definition.

\begin{defn}\label{def:division}
Given $x\in S$, we say that $\alpha\in S$ \emph{divides} $x$ if $x-\alpha\in S$. We denote by $\mathrm D(x)=\{\alpha\in S\mid x-\alpha \in S\}$ the set of \emph{divisors} of $x$. 
\end{defn}
\begin{defn}\label{FRdist}
Let $S$ be a numerical semigroup. For $m_1\in S$, let 
$\nu(m_1):=\sharp \mathrm D(m_1)$. The (classical) {\em Feng-Rao distance} of $S$ is defined
by the function
\[
\begin{array}{rcl}
\delta_{FR}\;:\;S&\longrightarrow&\N \\
m&\mapsto&\delta_{FR}(m):=\min\{\nu(m_1)\mid m_1\geq m,\;\;m_1\in S\}.
\end{array}
\]
\end{defn}

There are some well-known facts about the functions $\nu$ and $\delta_{FR}$ for an arbitrary semigroup $S$
(see \cite{HvLP}, \cite{KirPel} or \cite{WSPink} for further details).
An important one is that
$\delta_{FR}(m)\geq m+1-2g$ for all $m\in S$ with $m\geq c$\/,
and that equality holds if moreover $m\geq 2c-1$ (see also Proposition~\ref{prop:well-known}).

The classical Feng-Rao distance corresponds to $r=1$ in the following definition.

\begin{defn}\label{FRgen}
Let $S$ be a numerical semigroup.
For any set of distinct $m_1,\ldots,m_r\in S$\/, let
$\nu(m_1,\ldots,m_r):=\sharp \mathrm D(m_1,\ldots,m_r)$, 
where $\mathrm D(m_1,\ldots,m_r):=\mathrm D(m_1)\cup\cdots\cup \mathrm D(m_r)$\/. 

For any integer $r\geq 1$,
the {\em $r$\/-th Feng-Rao distance} of $S$ is defined by the function
\[
\begin{array}{rcl}
{\delta_{FR}^{r}}\;:\;S&\longrightarrow&\N \\
m&\mapsto&\delta_{FR}^{r}(m)
\end{array}\]
where
$
\delta_{FR}^{r}(m)=
\min\{\nu(m_1,\ldots,m_r)\mid m\leq m_{1}<\cdots<m_{r},\;\;m_{i}\in S\}.
$
\end{defn}

Very few results are known for the numbers $\delta_{FR}^{r}$\/,
and their computation is very hard from both a theoretical and
computational point of view. The main result we need describes
the asymptotical behavior for $m>>0$, and was proven in \cite{WCC}.
This result tells us that there exists a certain constant $E_{r}=\mathrm E(S,r)$\/,
depending on $r$ and $S$, such that
\[
\delta_{FR}^{r}(m)=m+1-2g+E_{r}
\]
for $m\geq 2c-1$.
\begin{defn}\label{FRnumber}
This constant $\textrm E(S,r)$ is called the {\em $r$-th Feng-Rao number} of the semigroup $S$\/.
\end{defn}
Furthermore,
it is also true that $\delta_{FR}^{r}(m)\geq m+1-2g+\mathrm E(S,r)$ for $m\geq c$
(see \cite{WCC}).

Note that, for any non-negative integer $k$ and $m \ge 2c-1$, $\delta_{FR}^{r}(m+k) = k + \delta_{FR}^{r}(m)$.

We will simplify the notation by writting $\delta^r(m)$ for $\delta_{FR}^r(m)$.

\begin{defn}\label{def:optimal_configuration}
Let $S$ be a numerical semigroup and let $m\in S$.
A finite subset of $S\cap [m,\infty)$ is called a $(S,m)$-\emph{configuration}, or simply a configuration. A configuration $M$ of cardinality $r$ is said to be \emph{optimal} if $\delta^r(m)=\sharp\mathrm D(M)$, where $\mathrm D(M):=\cup_{x\in M}\mathrm D(x)$\/. 
\end{defn}

\subsection{A convenient visualisation of the integers}\label{convenient.visualisation}
We can think of the integers as points disposed regularly on a cylindrical helix (Figure~\ref{fig:helix}). 

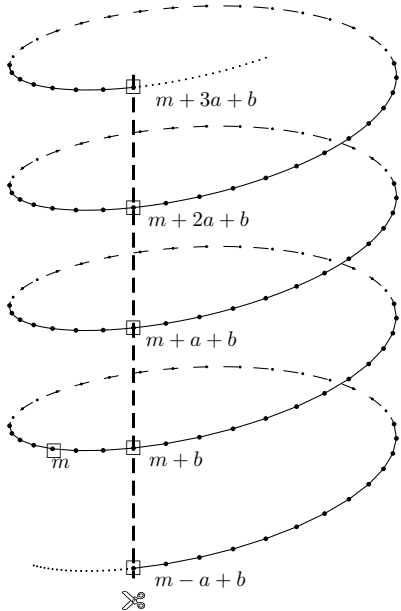
\begin{figure}[h]
\begin{center}
\begin{tikzpicture}[y=4cm/360,scale=.8]
\pgfplothandlerlineto
\pgfsetplotmarksize{.05ex}
\pgfplothandlermark{\pgfuseplotmark{*}}
\pgfplotfunction{\y}{-20,...,0}{\pgfpointxyz{3*sin(2*\y)}{\y}{3*cos(2*\y)}} 
\pgfplothandlerlineto
\pgfsetplotmarksize{.2ex}
\pgfplothandlermark{\pgfuseplotmark{*}}
\pgfplotfunction{\y}{0,5,...,60}{\pgfpointxyz{3*sin(2*\y)}{\y}{3*cos(2*\y)}} 
\pgfplothandlerlineto
\pgfsetplotmarksize{.1ex}
\pgfplotfunction{\y}{0,5,...,60}{\pgfpointxyz{3*sin(2*\y)}{\y}{3*cos(2*\y)}} 
\pgfusepath{stroke}
\pgfsetdash{{0.2cm}}{0.1cm}
\pgfplothandlerlineto
\pgfsetplotmarksize{.1ex}
\pgfplothandlermark{\pgfuseplotmark{*}}
\pgfplotfunction{\y}{60,65,...,145}{\pgfpointxyz{3*sin(2*\y)}{\y}{3*cos(2*\y)}} 
\pgfplothandlerlineto
\pgfplotfunction{\y}{60,...,145}{\pgfpointxyz{3*sin(2*\y)}{\y}{3*cos(2*\y)}} 
\pgfusepath{stroke}
\pgfsetdash{}{0.1cm}
\pgfplothandlerlineto
\pgfsetplotmarksize{.2ex}
\pgfplothandlermark{\pgfuseplotmark{*}}
\pgfplotfunction{\y}{145,150,...,240}{\pgfpointxyz{3*sin(2*\y)}{\y}{3*cos(2*\y)}} 
\pgfplothandlerlineto
\pgfsetplotmarksize{.1ex}
\pgfplotfunction{\y}{145,150,...,240}{\pgfpointxyz{3*sin(2*\y)}{\y}{3*cos(2*\y)}} 
\pgfusepath{stroke}
\pgfsetdash{{0.2cm}}{0.1cm}
\pgfplothandlerlineto
\pgfsetplotmarksize{.1ex}
\pgfplothandlermark{\pgfuseplotmark{*}}
\pgfplotfunction{\y}{240,245,...,325}{\pgfpointxyz{3*sin(2*\y)}{\y}{3*cos(2*\y)}} 
\pgfplothandlerlineto
\pgfsetdash{{0.2cm}}{0.1cm}
\pgfplotfunction{\y}{240,...,325}{\pgfpointxyz{3*sin(2*\y)}{\y}{3*cos(2*\y)}} 
\pgfusepath{stroke}
\pgfsetdash{}{0.1cm}
\pgfplothandlerlineto
\pgfsetplotmarksize{.2ex}
\pgfplothandlermark{\pgfuseplotmark{*}}
\pgfplotfunction{\y}{325,330,...,420}{\pgfpointxyz{3*sin(2*\y)}{\y}{3*cos(2*\y)}} 
\pgfplothandlerlineto
\pgfsetplotmarksize{.1ex}
\pgfplotfunction{\y}{325,...,420}{\pgfpointxyz{3*sin(2*\y)}{\y}{3*cos(2*\y)}} 
\pgfusepath{stroke}
\pgfsetdash{{0.2cm}}{0.1cm}
\pgfplothandlerlineto
\pgfsetplotmarksize{.1ex}
\pgfplothandlermark{\pgfuseplotmark{*}}
\pgfplotfunction{\y}{420,425,...,505}{\pgfpointxyz{3*sin(2*\y)}{\y}{3*cos(2*\y)}} 
\pgfplothandlerlineto
\pgfsetdash{{0.2cm}}{0.1cm}
\pgfplotfunction{\y}{420,...,505}{\pgfpointxyz{3*sin(2*\y)}{\y}{3*cos(2*\y)}} 
\pgfusepath{stroke}
\pgfsetdash{}{0.1cm}
\pgfplothandlerlineto
\pgfsetplotmarksize{.2ex}
\pgfplothandlermark{\pgfuseplotmark{*}}
\pgfplotfunction{\y}{505,510,...,600}{\pgfpointxyz{3*sin(2*\y)}{\y}{3*cos(2*\y)}} 
\pgfplothandlerlineto
\pgfsetplotmarksize{.1ex}
\pgfplotfunction{\y}{505,...,600}{\pgfpointxyz{3*sin(2*\y)}{\y}{3*cos(2*\y)}} 
\pgfusepath{stroke}
\pgfsetdash{{0.2cm}}{0.1cm}
\pgfplothandlerlineto
\pgfsetplotmarksize{.1ex}
\pgfplothandlermark{\pgfuseplotmark{*}}
\pgfplotfunction{\y}{600,605,...,690}{\pgfpointxyz{3*sin(2*\y)}{\y}{3*cos(2*\y)}} 
\pgfplothandlerlineto
\pgfsetdash{{0.2cm}}{0.1cm}
\pgfplotfunction{\y}{600,...,690}{\pgfpointxyz{3*sin(2*\y)}{\y}{3*cos(2*\y)}} 
\pgfusepath{stroke}
\pgfsetdash{}{0.1cm}
\pgfplothandlerlineto
\pgfsetplotmarksize{.2ex}
\pgfplothandlermark{\pgfuseplotmark{*}}
\pgfplotfunction{\y}{685,690,...,720}{\pgfpointxyz{3*sin(2*\y)}{\y}{3*cos(2*\y)}} 
\pgfplothandlerlineto
\pgfsetplotmarksize{.1ex}
\pgfplotfunction{\y}{685,...,720}{\pgfpointxyz{3*sin(2*\y)}{\y}{3*cos(2*\y)}} 
\pgfusepath{stroke}

\pgfplothandlerlineto
\pgfsetplotmarksize{.05ex}
\pgfplothandlermark{\pgfuseplotmark{*}}
\pgfplotfunction{\y}{720,...,740}{\pgfpointxyz{3*sin(2*\y)}{\y}{3*cos(2*\y)}} 

\pgfsetlinewidth{.4mm}
 \pgfsetdash{{0.2cm}{0.1cm}}{1cm}
 \pgfpathmoveto{\pgfpointxyz{0}{-15}{3}}
\pgfpathlineto{\pgfpointxyz{0}{745}{3}}
\pgfusepath{stroke}
\foreach \d in {0,1,2,3,4}
{\pgftext[at=
\pgfpointlineatdistance{\d*56.9}{\pgfpointxyz{0}{0}{3}}{\pgfpointxyz{0}{720}{3}}]{$\square$}}
\pgftext[at=
\pgfpointlineatdistance{0}{\pgfpointxyz{1.3}{0}{3.5}}{\pgfpointxyz{.9}{720}{3.5}}]{$m-a+b$}
\pgftext[at=
\pgfpointlineatdistance{56.9}{\pgfpointxyz{.9}{0}{3.5}}{\pgfpointxyz{.9}{720}{3.5}}]{$m+b$}
\pgftext[at=
\pgfpointlineatdistance{113.8}{\pgfpointxyz{.9}{0}{3.5}}{\pgfpointxyz{1.4}{720}{3.5}}]{$m+a+b$}
\pgftext[at=
\pgfpointlineatdistance{170.7}{\pgfpointxyz{.9}{0}{3.5}}{\pgfpointxyz{1.4}{720}{3.5}}]{$m+2a+b$}
\pgftext[at=
\pgfpointlineatdistance{227.6}{\pgfpointxyz{.9}{0}{3.5}}{\pgfpointxyz{1.4}{720}{3.5}}]{$m+3a+b$}

{\pgftext[at=
\pgfpointlineatdistance{56.9}{\pgfpointxyz{-1.75}{0}{3}}{\pgfpointxyz{0}{720}{3}}]{$\square$}}
\pgftext[at=
\pgfpointlineatdistance{56.9}{\pgfpointxyz{-1.5}{0}{3.5}}{\pgfpointxyz{.5}{720}{3.5}}]{$m$}

\pgftext[at=
\pgfpointlineatdistance{-15}{\pgfpointxyz{0}{0}{3}}{\pgfpointxyz{0}{720}{3}}]{\ScissorHollowLeft}
\end{tikzpicture}
\end{center}
\caption{The integers on an helix}\label{fig:helix}
\end{figure}

As using the sketch of Figure~\ref{fig:helix} some of the integers would be hidden, we will consider planifications of the cylinder instead. They are usually obtained by cutting the cylinder through a vertical line passing through a point previously chosen. Note that a planification corresponds to taking a partition of the integers.
The reader may think on the letters $m, a,b$ as being $2c-1$, $n_1$ and $n_e-n_1$ for a semigroup generated by $\{n_1<\cdots<n_e\}$ whose conductor is $c$. This will be the case when dealing with semigroups generated by intervals.

We shall use this drawings to depict the most relevant parts of the sets considered. For instance, if we want to highlight the elements of a numerical semigroup, we do not add any information by depicting the points below $0$ and those above the conductor.

The parallelograms in Figure~\ref{fig:sgp9-13-15} highlight the elements of the semigroup $S= \langle 9,13,15\rangle$, and the elements of $60-S$, respectively.
\begin{center}
\begin{figure}[h]
\begin{tikzpicture}
[first/.style={circle,draw=black!80,fill=red!80,thick,
inner sep=0pt,minimum size=4.5000000000mm}
,second/.style={circle,draw=black!50,fill=blue!50,thick, 
inner sep=0pt,minimum size=4.5000000000mm},
intersection/.style={circle,draw=black!80,fill=black!60,thick, 
inner sep=0pt,minimum size=4.5000000000mm}, 
gap/.style={circle,draw=black!40,fill=green!30,thick,
 inner sep=0pt,minimum size=4.5000000000mm}]
\node at (0,3) [first] {\tiny 54};
\node at (0.5000000000,3.0555555555) [first] {\tiny 55};
\node at (1,3.1111111111) [first] {\tiny 56};
\node at (1.5000000000,3.1666666666) [first] {\tiny 57};
\node at (2,3.2222222222) [first] {\tiny 58};
\node at (2.5000000000,3.2777777777) [first] {\tiny 59};
\node at (3,3.3333333333) [first] {\tiny 60};
\node at (3.5000000000,3.3888888888) [first] {\tiny 61};
\node at (4,3.4444444444) [first] {\tiny 62};
\node at (0,2.5000000000) [first] {\tiny 45};
\node at (0.5000000000,2.5555555555) [first] {\tiny 46};
\node at (1,2.6111111111) [gap] {\tiny 47};
\node at (1.5000000000,2.6666666666) [first] {\tiny 48};
\node at (2,2.7222222222) [first] {\tiny 49};
\node at (2.5000000000,2.7777777777) [first] {\tiny 50};
\node at (3,2.8333333333) [first] {\tiny 51};
\node at (3.5000000000,2.8888888888) [first] {\tiny 52};
\node at (4,2.9444444444) [first] {\tiny 53};
\node at (0,2) [first] {\tiny 36};
\node at (0.5000000000,2.0555555555) [first] {\tiny 37};
\node at (1,2.1111111111) [gap] {\tiny 38};
\node at (1.5000000000,2.1666666666) [first] {\tiny 39};
\node at (2,2.2222222222) [first] {\tiny 40};
\node at (2.5000000000,2.2777777777) [first] {\tiny 41};
\node at (3,2.3333333333) [first] {\tiny 42};
\node at (3.5000000000,2.3888888888) [first] {\tiny 43};
\node at (4,2.4444444444) [first] {\tiny 44};
\node at (0,1.5000000000) [first] {\tiny 27};
\node at (0.5000000000,1.5555555555) [first] {\tiny 28};
\node at (1,1.6111111111) [gap] {\tiny 29};
\node at (1.5000000000,1.6666666666) [first] {\tiny 30};
\node at (2,1.7222222222) [first] {\tiny 31};
\node at (2.5000000000,1.7777777777) [gap] {\tiny 32};
\node at (3,1.8333333333) [first] {\tiny 33};
\node at (3.5000000000,1.8888888888) [gap] {\tiny 34};
\node at (4,1.9444444444) [first] {\tiny 35};
\node at (0,1) [first] {\tiny 18};
\node at (0.5000000000,1.0555555555) [gap] {\tiny 19};
\node at (1,1.1111111111) [gap] {\tiny 20};
\node at (1.5000000000,1.1666666666) [gap] {\tiny 21};
\node at (2,1.2222222222) [first] {\tiny 22};
\node at (2.5000000000,1.2777777777) [gap] {\tiny 23};
\node at (3,1.3333333333) [first] {\tiny 24};
\node at (3.5000000000,1.3888888888) [gap] {\tiny 25};
\node at (4,1.4444444444) [first] {\tiny 26};
\node at (0,0.5000000000) [first] {\tiny 9};
\node at (0.5000000000,0.5555555555) [gap] {\tiny 10};
\node at (1,0.6111111111) [gap] {\tiny 11};
\node at (1.5000000000,0.6666666666) [gap] {\tiny 12};
\node at (2,0.7222222222) [first] {\tiny 13};
\node at (2.5000000000,0.7777777777) [gap] {\tiny 14};
\node at (3,0.8333333333) [first] {\tiny 15};
\node at (3.5000000000,0.8888888888) [gap] {\tiny 16};
\node at (4,0.9444444444) [gap] {\tiny 17};
\node at (0,0) [first] {\tiny 0};
\node at (0.5000000000,0.0555555555) [gap] {\tiny 1};
\node at (1,0.1111111111) [gap] {\tiny 2};
\node at (1.5000000000,0.1666666666) [gap] {\tiny 3};
\node at (2,0.2222222222) [gap] {\tiny 4};
\node at (2.5000000000,0.2777777777) [gap] {\tiny 5};
\node at (3,0.3333333333) [gap] {\tiny 6};
\node at (3.5000000000,0.3888888888) [gap] {\tiny 7};
\node at (4,0.4444444444) [gap] {\tiny 8};
\end{tikzpicture}
\quad
\begin{tikzpicture}
[first/.style={circle,draw=black!80,fill=red!80,thick,
inner sep=0pt,minimum size=4.5000000000mm}
,second/.style={circle,draw=black!50,fill=blue!50,thick, 
inner sep=0pt,minimum size=4.5000000000mm},
intersection/.style={circle,draw=black!80,fill=black!60,thick, 
inner sep=0pt,minimum size=4.5000000000mm}, 
gap/.style={circle,draw=black!40,fill=green!30,thick,
 inner sep=0pt,minimum size=4.5000000000mm}]
\node at (0,3) [gap] {\tiny 54};
\node at (0.5000000000,3.0555555555) [gap] {\tiny 55};
\node at (1,3.1111111111) [gap] {\tiny 56};
\node at (1.5000000000,3.1666666666) [gap] {\tiny 57};
\node at (2,3.2222222222) [gap] {\tiny 58};
\node at (2.5000000000,3.2777777777) [gap] {\tiny 59};
\node at (3,3.3333333333) [first] {\tiny 60};
\node at (3.5000000000,3.3888888888) [gap] {\tiny 61};
\node at (4,3.4444444444) [gap] {\tiny 62};
\node at (0,2.5000000000) [first] {\tiny 45};
\node at (0.5000000000,2.5555555555) [gap] {\tiny 46};
\node at (1,2.6111111111) [first] {\tiny 47};
\node at (1.5000000000,2.6666666666) [gap] {\tiny 48};
\node at (2,2.7222222222) [gap] {\tiny 49};
\node at (2.5000000000,2.7777777777) [gap] {\tiny 50};
\node at (3,2.8333333333) [first] {\tiny 51};
\node at (3.5000000000,2.8888888888) [gap] {\tiny 52};
\node at (4,2.9444444444) [gap] {\tiny 53};
\node at (0,2) [first] {\tiny 36};
\node at (0.5000000000,2.0555555555) [gap] {\tiny 37};
\node at (1,2.1111111111) [first] {\tiny 38};
\node at (1.5000000000,2.1666666666) [gap] {\tiny 39};
\node at (2,2.2222222222) [gap] {\tiny 40};
\node at (2.5000000000,2.2777777777) [gap] {\tiny 41};
\node at (3,2.3333333333) [first] {\tiny 42};
\node at (3.5000000000,2.3888888888) [gap] {\tiny 43};
\node at (4,2.4444444444) [gap] {\tiny 44};
\node at (0,1.5000000000) [first] {\tiny 27};
\node at (0.5000000000,1.5555555555) [gap] {\tiny 28};
\node at (1,1.6111111111) [first] {\tiny 29};
\node at (1.5000000000,1.6666666666) [first] {\tiny 30};
\node at (2,1.7222222222) [gap] {\tiny 31};
\node at (2.5000000000,1.7777777777) [first] {\tiny 32};
\node at (3,1.8333333333) [first] {\tiny 33};
\node at (3.5000000000,1.8888888888) [first] {\tiny 34};
\node at (4,1.9444444444) [gap] {\tiny 35};
\node at (0,1) [first] {\tiny 18};
\node at (0.5000000000,1.0555555555) [first] {\tiny 19};
\node at (1,1.1111111111) [first] {\tiny 20};
\node at (1.5000000000,1.1666666666) [first] {\tiny 21};
\node at (2,1.2222222222) [gap] {\tiny 22};
\node at (2.5000000000,1.2777777777) [first] {\tiny 23};
\node at (3,1.3333333333) [first] {\tiny 24};
\node at (3.5000000000,1.3888888888) [first] {\tiny 25};
\node at (4,1.4444444444) [gap] {\tiny 26};
\node at (0,0.5000000000) [first] {\tiny 9};
\node at (0.5000000000,0.5555555555) [first] {\tiny 10};
\node at (1,0.6111111111) [first] {\tiny 11};
\node at (1.5000000000,0.6666666666) [first] {\tiny 12};
\node at (2,0.7222222222) [gap] {\tiny 13};
\node at (2.5000000000,0.7777777777) [first] {\tiny 14};
\node at (3,0.8333333333) [first] {\tiny 15};
\node at (3.5000000000,0.8888888888) [first] {\tiny 16};
\node at (4,0.9444444444) [first] {\tiny 17};
\node at (0,0) [first] {\tiny 0};
\node at (0.5000000000,0.0555555555) [first] {\tiny 1};
\node at (1,0.1111111111) [first] {\tiny 2};
\node at (1.5000000000,0.1666666666) [first] {\tiny 3};
\node at (2,0.2222222222) [first] {\tiny 4};
\node at (2.5000000000,0.2777777777) [first] {\tiny 5};
\node at (3,0.3333333333) [first] {\tiny 6};
\node at (3.5000000000,0.3888888888) [first] {\tiny 7};
\node at (4,0.4444444444) [first] {\tiny 8};
\end{tikzpicture}
\caption{The semigroup $S= \langle 9,13,15\rangle$ and $60-S$, respectively}\label{fig:sgp9-13-15}
\end{figure}
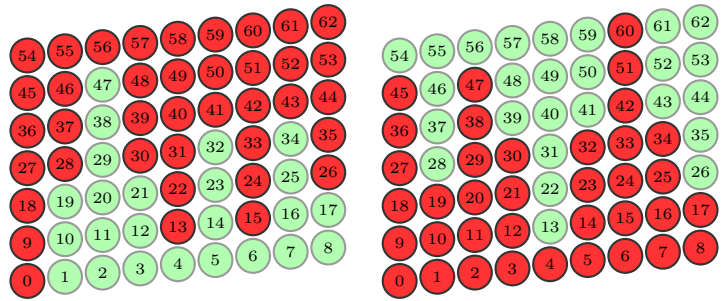
\end{center}
Most times we are interested in finite sets of integers which are non smaller than a given integer $m$. In this case we  prefer to draw all the points from $m$ to $m+a+b$ at the same level. See Figure~\ref{fig:first_amenable} for an example. Its caption will soon become clear. For convenience, the columns are numbered. Having such a picture in mind, we can think on a partition of the set of integers greater than $m$ whose classes are the columns (the $i$th column of a set is the set its elements congruent with $i$ modulo $a$). 
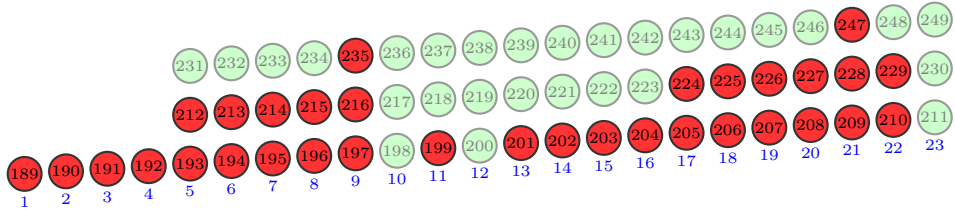
\begin{figure}[h]
\begin{tikzpicture}
[first/.style={circle,draw=black!80,fill=red!80,thick,
inner sep=0pt,minimum size=4.5000000000mm},
trg/.style={diamond,draw=black!80,fill=red!50,thick,
inner sep=0pt,minimum size=4.5000000000mm},
rect/.style={rectangle,draw=black!70,fill=green!40,thick,
inner sep=0pt,minimum size=4.5000000000mm},
wagon/.style={rectangle,draw=black!85,fill=blue!40,thick,
inner sep=0pt,minimum size=4.5000000000mm},
pivot/.style={star,draw=black!80,fill=blue!60,thick,
inner sep=0pt,minimum size=4.5000000000mm},
ground/.style={circle,draw=black!50,fill=blue!30,thick, 
inner sep=0pt,minimum size=4.5000000000mm},
int_triang/.style={diamond,draw=black!80,fill=red!80,thick, 
inner sep=0pt,minimum size=4.5000000000mm}, 
int_rect/.style={rectangle,draw=black!80,fill=black!10!green!90,thick, 
inner sep=0pt,minimum size=4.5000000000mm}, 
int_wagon/.style={rectangle,draw=black!85,fill=blue!80,thick, 
inner sep=0pt,minimum size=4.5000000000mm}, 
gap/.style={circle,draw=black!40,fill=green!20,thick,
 inner sep=0pt,minimum size=4.5000000000mm}]
\node at (2.2000000000,1.4368421052) [gap] {{\color{gray}{\tiny 231}}};
\node at (2.7500000000,1.4710526315) [gap] {{\color{gray}{\tiny 232}}};
\node at (3.3000000000,1.5052631578) [gap] {{\color{gray}{\tiny 233}}};
\node at (3.8500000000,1.5394736842) [gap] {{\color{gray}{\tiny 234}}};
\node at (4.4000000000,1.5736842105) [first] {\tiny 235};
\node at (4.9500000000,1.6078947368) [gap] {{\color{gray}{\tiny 236}}};
\node at (5.5000000000,1.6421052631) [gap] {{\color{gray}{\tiny 237}}};
\node at (6.0500000000,1.6763157894) [gap] {{\color{gray}{\tiny 238}}};
\node at (6.6000000000,1.7105263157) [gap] {{\color{gray}{\tiny 239}}};
\node at (7.1500000000,1.7447368421) [gap] {{\color{gray}{\tiny 240}}};
\node at (7.7000000000,1.7789473684) [gap] {{\color{gray}{\tiny 241}}};
\node at (8.2500000000,1.8131578947) [gap] {{\color{gray}{\tiny 242}}};
\node at (8.8000000000,1.8473684210) [gap] {{\color{gray}{\tiny 243}}};
\node at (9.3500000000,1.8815789473) [gap] {{\color{gray}{\tiny 244}}};
\node at (9.9000000000,1.9157894736) [gap] {{\color{gray}{\tiny 245}}};
\node at (10.4500000000,1.9500000000) [gap] {{\color{gray}{\tiny 246}}};
\node at (11,1.9842105263) [first] {\tiny 247};
\node at (11.5500000000,2.0184210526) [gap] {{\color{gray}{\tiny 248}}};
\node at (12.1000000000,2.0526315789) [gap] {{\color{gray}{\tiny 249}}};
\node at (2.2000000000,0.7868421052) [first] {\tiny 212};
\node at (2.7500000000,0.8210526315) [first] {\tiny 213};
\node at (3.3000000000,0.8552631578) [first] {\tiny 214};
\node at (3.8500000000,0.8894736842) [first] {\tiny 215};
\node at (4.4000000000,0.9236842105) [first] {\tiny 216};
\node at (4.9500000000,0.9578947368) [gap] {{\color{gray}{\tiny 217}}};
\node at (5.5000000000,0.9921052631) [gap] {{\color{gray}{\tiny 218}}};
\node at (6.0500000000,1.0263157894) [gap] {{\color{gray}{\tiny 219}}};
\node at (6.6000000000,1.0605263157) [gap] {{\color{gray}{\tiny 220}}};
\node at (7.1500000000,1.0947368421) [gap] {{\color{gray}{\tiny 221}}};
\node at (7.7000000000,1.1289473684) [gap] {{\color{gray}{\tiny 222}}};
\node at (8.2500000000,1.1631578947) [gap] {{\color{gray}{\tiny 223}}};
\node at (8.8000000000,1.1973684210) [first] {\tiny 224};
\node at (9.3500000000,1.2315789473) [first] {\tiny 225};
\node at (9.9000000000,1.2657894736) [first] {\tiny 226};
\node at (10.4500000000,1.3000000000) [first] {\tiny 227};
\node at (11,1.3342105263) [first] {\tiny 228};
\node at (11.5500000000,1.3684210526) [first] {\tiny 229};
\node at (12.1000000000,1.4026315789) [gap] {{\color{gray}{\tiny 230}}};
\node at (0,0) [first] {\tiny 189};
\node at (0.5500000000,0.0342105263) [first] {\tiny 190};
\node at (1.1000000000,0.0684210526) [first] {\tiny 191};
\node at (1.6500000000,0.1026315789) [first] {\tiny 192};
\node at (2.2000000000,0.1368421052) [first] {\tiny 193};
\node at (2.7500000000,0.1710526315) [first] {\tiny 194};
\node at (3.3000000000,0.2052631578) [first] {\tiny 195};
\node at (3.8500000000,0.2394736842) [first] {\tiny 196};
\node at (4.4000000000,0.2736842105) [first] {\tiny 197};
\node at (4.9500000000,0.3078947368) [gap] {{\color{gray}{\tiny 198}}};
\node at (5.5000000000,0.3421052631) [first] {\tiny 199};
\node at (6.0500000000,0.3763157894) [gap] {{\color{gray}{\tiny 200}}};
\node at (6.6000000000,0.4105263157) [first] {\tiny 201};
\node at (7.1500000000,0.4447368421) [first] {\tiny 202};
\node at (7.7000000000,0.4789473684) [first] {\tiny 203};
\node at (8.2500000000,0.5131578947) [first] {\tiny 204};
\node at (8.8000000000,0.5473684210) [first] {\tiny 205};
\node at (9.3500000000,0.5815789473) [first] {\tiny 206};
\node at (9.9000000000,0.6157894736) [first] {\tiny 207};
\node at (10.4500000000,0.6500000000) [first] {\tiny 208};
\node at (11,0.6842105263) [first] {\tiny 209};
\node at (11.5500000000,0.7184210526) [first] {\tiny 210};
\node at (12.1000000000,0.7526315789) [gap] {{\color{gray}{\tiny 211}}};
\node at (0,0) [below=5pt] {{\color{blue}{\tiny 1}}};
\node at (0.5500000000,0.0342105263) [below=5pt] {{\color{blue}{\tiny 2}}};
\node at (1.1000000000,0.0684210526) [below=5pt] {{\color{blue}{\tiny 3}}};
\node at (1.6500000000,0.1026315789) [below=5pt] {{\color{blue}{\tiny 4}}};
\node at (2.2000000000,0.1368421052) [below=5pt] {{\color{blue}{\tiny 5}}};
\node at (2.7500000000,0.1710526315) [below=5pt] {{\color{blue}{\tiny 6}}};
\node at (3.3000000000,0.2052631578) [below=5pt] {{\color{blue}{\tiny 7}}};
\node at (3.8500000000,0.2394736842) [below=5pt] {{\color{blue}{\tiny 8}}};
\node at (4.4000000000,0.2736842105) [below=5pt] {{\color{blue}{\tiny 9}}};
\node at (4.9500000000,0.3078947368) [below=5pt] {{\color{blue}{\tiny 10}}};
\node at (5.5000000000,0.3421052631) [below=5pt] {{\color{blue}{\tiny 11}}};
\node at (6.0500000000,0.3763157894) [below=5pt] {{\color{blue}{\tiny 12}}};
\node at (6.6000000000,0.4105263157) [below=5pt] {{\color{blue}{\tiny 13}}};
\node at (7.1500000000,0.4447368421) [below=5pt] {{\color{blue}{\tiny 14}}};
\node at (7.7000000000,0.4789473684) [below=5pt] {{\color{blue}{\tiny 15}}};
\node at (8.2500000000,0.5131578947) [below=5pt] {{\color{blue}{\tiny 16}}};
\node at (8.8000000000,0.5473684210) [below=5pt] {{\color{blue}{\tiny 17}}};
\node at (9.3500000000,0.5815789473) [below=5pt] {{\color{blue}{\tiny 18}}};
\node at (9.9000000000,0.6157894736) [below=5pt] {{\color{blue}{\tiny 19}}};
\node at (10.4500000000,0.6500000000) [below=5pt] {{\color{blue}{\tiny 20}}};
\node at (11,0.6842105263) [below=5pt] {{\color{blue}{\tiny 21}}};
\node at (11.5500000000,0.7184210526) [below=5pt] {{\color{blue}{\tiny 22}}};
\node at (12.1000000000,0.7526315789) [below=5pt] {{\color{blue}{\tiny 23}}};
\end{tikzpicture}
\caption{An amenable set}\label{fig:first_amenable}
\end{figure}
\section{A generic algorithm}\label{sec:algorithm_fr}
We shall start the section by giving a quite efficient algorithm to compute the divisors of an element of a numerical semigroup.
The aim is then to find an optimal configuration. Note that if $M$ is wanted to be an optimal configuration, 
we just have to control the cardinality of the difference $\mathrm D(M)\setminus \mathrm D(m)$, for all possible $m\in M$.

Among the optimal configurations there is an amenable set (Proposition~\ref{condiciones-ms}). 
Thus, one can search for an optimal configuration among the amenable sets, 
which can be constructed using Algorithm~\ref{alg:compute_amenables}.
Due to the results in Section~\ref{subsec:ground} (Corollary~\ref{cor:shadow_divisors}, to be more specific), 
one only needs to consider one amenable set for each shadow.

\subsection{Divisors}\label{subsec:divisors}

Recall that given $x\in S$, we say that $\alpha\in S$ divides $x$ if $x-\alpha\in S$. We denote by $\mathrm D(x)$ the set of \emph{divisors} of $x$. 

Note that $\mathrm D(x)\subseteq[0,x]$ and $s\in \mathrm D(x)$ implies $\mathrm D(s)\subseteq \mathrm D(x)$\/.

\begin{lemma}\label{lema:divisors}
$\mathrm D(x)=S\cap (x-S)$.
\end{lemma}
\begin{proof}
Let $\alpha\in \mathrm D(x)$. By definition, $\alpha\in S$ and $x-\alpha\in S$. But then $\alpha= x-(x-\alpha)\in x-S$.

Conversely, let $\alpha\in S$ be such that there exists $\beta\in S$ for which $x-\beta=\alpha$. But then  $x-\alpha=\beta\in S$, proving that $\alpha$ divides $x$.
\end{proof}
We observe that elements greater than $x$ need not to be used to compute the divisors of $x$. Denoting $S_x=\{n\in S\mid n\le x\}$, we get the following:
\begin{corollary}\label{cor:divisors}
$\mathrm D(x)=S_x\cap (x-S_x)$.
\end{corollary}
The computation of the divisors of an element can be easily implemented (Algorithm~\ref{alg:compute_divisors}) due to this consequence of Lemma~\ref{lema:divisors}. Note also that, once we compute the elements of $S$ smaller than $x$ (which can easily be done if the conductor is known), the computation of the divisors is immediate.

\begin{algorithm}[h]\caption{Divisors}
\label{alg:compute_divisors}
\DontPrintSemicolon
\SetKwInOut{Input}{Input}\SetKwInOut{Output}{Output}
\Input{A numerical semigroup $S$, $x\in S$}
\Output{The divisors of $x$}
\BlankLine
\nl $S_x := \{s\in S\mid s\le x\}$\tcc{Compute the elements of $S$ smaller than $x$}
\nl return $\{s\in S_x\mid x-s\in S_x\}$
\end{algorithm}

The highlighted elements in Figure~\ref{fig:divs60sgp9-13-15} represent the divisors of $60\in \langle 9,13,15\rangle$. They are obtained intersecting the highlighted elements of the pictures in Figure~\ref{fig:sgp9-13-15}.
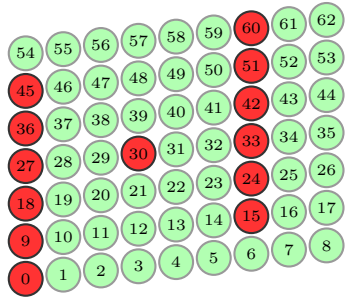
\begin{figure}[h]
\begin{center}
\begin{tikzpicture}
[first/.style={circle,draw=black!80,fill=red!80,thick,
inner sep=0pt,minimum size=4.5000000000mm}
,second/.style={circle,draw=black!50,fill=blue!50,thick, 
inner sep=0pt,minimum size=4.5000000000mm},
intersection/.style={circle,draw=black!80,fill=black!60,thick, 
inner sep=0pt,minimum size=4.5000000000mm}, 
gap/.style={circle,draw=black!40,fill=green!30,thick,
 inner sep=0pt,minimum size=4.5000000000mm}]
\node at (0,3) [gap] {\tiny 54};
\node at (0.5000000000,3.0555555555) [gap] {\tiny 55};
\node at (1,3.1111111111) [gap] {\tiny 56};
\node at (1.5000000000,3.1666666666) [gap] {\tiny 57};
\node at (2,3.2222222222) [gap] {\tiny 58};
\node at (2.5000000000,3.2777777777) [gap] {\tiny 59};
\node at (3,3.3333333333) [first] {\tiny 60};
\node at (3.5000000000,3.3888888888) [gap] {\tiny 61};
\node at (4,3.4444444444) [gap] {\tiny 62};
\node at (0,2.5000000000) [first] {\tiny 45};
\node at (0.5000000000,2.5555555555) [gap] {\tiny 46};
\node at (1,2.6111111111) [gap] {\tiny 47};
\node at (1.5000000000,2.6666666666) [gap] {\tiny 48};
\node at (2,2.7222222222) [gap] {\tiny 49};
\node at (2.5000000000,2.7777777777) [gap] {\tiny 50};
\node at (3,2.8333333333) [first] {\tiny 51};
\node at (3.5000000000,2.8888888888) [gap] {\tiny 52};
\node at (4,2.9444444444) [gap] {\tiny 53};
\node at (0,2) [first] {\tiny 36};
\node at (0.5000000000,2.0555555555) [gap] {\tiny 37};
\node at (1,2.1111111111) [gap] {\tiny 38};
\node at (1.5000000000,2.1666666666) [gap] {\tiny 39};
\node at (2,2.2222222222) [gap] {\tiny 40};
\node at (2.5000000000,2.2777777777) [gap] {\tiny 41};
\node at (3,2.3333333333) [first] {\tiny 42};
\node at (3.5000000000,2.3888888888) [gap] {\tiny 43};
\node at (4,2.4444444444) [gap] {\tiny 44};
\node at (0,1.5000000000) [first] {\tiny 27};
\node at (0.5000000000,1.5555555555) [gap] {\tiny 28};
\node at (1,1.6111111111) [gap] {\tiny 29};
\node at (1.5000000000,1.6666666666) [first] {\tiny 30};
\node at (2,1.7222222222) [gap] {\tiny 31};
\node at (2.5000000000,1.7777777777) [gap] {\tiny 32};
\node at (3,1.8333333333) [first] {\tiny 33};
\node at (3.5000000000,1.8888888888) [gap] {\tiny 34};
\node at (4,1.9444444444) [gap] {\tiny 35};
\node at (0,1) [first] {\tiny 18};
\node at (0.5000000000,1.0555555555) [gap] {\tiny 19};
\node at (1,1.1111111111) [gap] {\tiny 20};
\node at (1.5000000000,1.1666666666) [gap] {\tiny 21};
\node at (2,1.2222222222) [gap] {\tiny 22};
\node at (2.5000000000,1.2777777777) [gap] {\tiny 23};
\node at (3,1.3333333333) [first] {\tiny 24};
\node at (3.5000000000,1.3888888888) [gap] {\tiny 25};
\node at (4,1.4444444444) [gap] {\tiny 26};
\node at (0,0.5000000000) [first] {\tiny 9};
\node at (0.5000000000,0.5555555555) [gap] {\tiny 10};
\node at (1,0.6111111111) [gap] {\tiny 11};
\node at (1.5000000000,0.6666666666) [gap] {\tiny 12};
\node at (2,0.7222222222) [gap] {\tiny 13};
\node at (2.5000000000,0.7777777777) [gap] {\tiny 14};
\node at (3,0.8333333333) [first] {\tiny 15};
\node at (3.5000000000,0.8888888888) [gap] {\tiny 16};
\node at (4,0.9444444444) [gap] {\tiny 17};
\node at (0,0) [first] {\tiny 0};
\node at (0.5000000000,0.0555555555) [gap] {\tiny 1};
\node at (1,0.1111111111) [gap] {\tiny 2};
\node at (1.5000000000,0.1666666666) [gap] {\tiny 3};
\node at (2,0.2222222222) [gap] {\tiny 4};
\node at (2.5000000000,0.2777777777) [gap] {\tiny 5};
\node at (3,0.3333333333) [gap] {\tiny 6};
\node at (3.5000000000,0.3888888888) [gap] {\tiny 7};
\node at (4,0.4444444444) [gap] {\tiny 8};
\end{tikzpicture}

\end{center}
\caption{The divisors of $60$ in the semigroup $S= \langle 9,13,15\rangle$}\label{fig:divs60sgp9-13-15}
\end{figure}

Another immediate consequence of Lemma~\ref{lema:divisors}, which has interest in concrete implementations, is the following corollary: 
\begin{corollary}\label{cor:big_divisors}
If $c\le x\le y$, then $\mathrm D(y)\cap [x,\infty)= (y-S) \cap [x,\infty)$.
\end{corollary}

We remember that 
\[
\mathrm D(m_1,\ldots,m_r)=\mathrm D(m_1)\cup\cdots\cup \mathrm D(m_r)=
\{p\in S\;|\;m_{i}-p\in S\;\;\mbox{for some } i\in\{1,\ldots,r\}\}
\]

The highlighted elements in Figure~\ref{fig:first_amenable} are the elements of $\mathrm D(235,199, 247, 229)$  which are greater than $189$, when $S$ is the semigroup $\langle 19,20,21,22,23\rangle$.

Observe that $x-S$ contains all the integers not greater than $x-c$ and that the number of integers smaller than $x$ not belonging to $x-S$ is precisely the genus of $S$. As the number of non-negative integers not greater than $x$ is $x+1$, one gets immediately the well known fact (see \cite{HvLP}, \cite{KirPel} or \cite{WSPink}): 
\begin{proposition}\label{prop:well-known}
If $x\ge 2c-1$, then $\sharp\mathrm D(x)=\sharp S\cap (x-S)=x+1-2g$.
\end{proposition}


\subsection{Amenable sets}\label{subsec:amenable}\begin{defn}\label{def:amenable}
Let $S$ be a numerical semigroup with conductor $c$. 
Let $M = \{m_1,\ldots,m_r\}\subseteq S$ with $2c-1 \le m = m_1<\cdots<m_r$. 
We say that the set $M$ is $(S,m,r)$-\emph{amenable} if: 
\begin{equation}\label{eq:def_amenable}
 \mbox{for all }i\in\{1,\ldots,r\},\mathrm D(m_i)\cap [m,\infty)\subseteq M.
\end{equation}
\end{defn}
We will refer a set satisfying (\ref{eq:def_amenable}) as being $m$-\emph{closed under division}. So, a subset of $S\cap[m,\infty)$ with cardinality $r$ is $(S,m,r)$-amenable if and only if it contains $m$ and is $m$-closed under division.

As a convention, the empty set is considered an $(S,m,0)$-amenable set, for any $m$.
When no confusion arises or only the concept is important, we say $(m,r)$-\emph{amenable} set or simply \emph{amenable} set.
\begin{example}
\begin{enumerate}
\item
Let $S=\langle 19, 20, 21, 22, 23 \rangle$. Its conductor is $c= 95$. Take $m=2c-1=189$. The set $M$ consisting of the highlighted elements in Figure~\ref{fig:first_amenable} is an amenable subset of $S$. 
\item
Let $S$ be a numerical semigroup with conductor $c$. Let $m\geq 2c-1$, and $r$ a non negative integer. Then the interval $[ m, m+r-1]\cap\N$ is a $(S,m,r)$-amenable set.
\end{enumerate}
\end{example}

The importance of amenable sets comes from the following result, which states that among the optimal configurations of cardinality $r$ there is at least one $(S,m,r)$-amenable set.

\begin{proposition}\label{condiciones-ms}
Let $S$ be a numerical semigroup with conductor $c$ and let $m\ge 2c-1$. Let $r$ be a positive integer. Among the optimal configurations of cardinality $r$ there is one $(S,m,r)$-amenable set.
\end{proposition}
\begin{proof}
Let $M=\{m_1,\ldots,m_r\}$ be an optimal configuration. As $m\geq 2c-1$, $\delta^r(m)$ is strictly increasing in $m$,
and thus $m$ cannot be less than $m_{1}$, which implies that $m_1=m$.

If $M$ is not $m$-closed under division, we may assume that for some $i\in \{1,\ldots,r\}$ there exists $t\in S$ such that $m_i-t > m$ and $m_i-t\not \in \{m_1,\ldots,m_r\}$. Clearly $\mathrm D(m_i-t)\subset \mathrm D(m_i)$, and thus $\mathrm D(m_1,\ldots,m_{i-1},m_i-t,m_{i+1},\ldots,m_r)\subseteq \mathrm D(m_1,\ldots,m_r)$. In other words, we can change $m_i$ by $m_i-t$ and the number of divisors does not increase.
Now we can repeat the process with the set obtained until we reach a $m$-closed under division set. Note that this must happen in a finite number of steps ($\N^r$ has no infinite descending chains).
\end{proof}
The definition of amenable set, which seems to be suitable for proofs, does not seem to help very much to do computations unless we can prove some consequences. The following one, showing that the distances between elements is somehow controlled, guarantees that the search of the amenable sets can be done in a bounded subset of $S$, and therefore amenable sets can be effectively computed. An algorithm will be presented (Algorithm \ref{alg:compute_amenables}).

\begin{proposition}\label{prop:amenable_first_consequences}
 Let $S$ be a numerical semigroup with conductor $c$ and let $m\ge 2c-1$.
 Let $M=\{m_1,\ldots,m_r\}\subseteq S$ be an $(S,m,r)$-amenable set and suppose that $S=\{0=\rho_1<\rho_2<\cdots\}$. Then
\begin{enumerate}[(a)]
\item $m_i\leq m+\rho_i$, for all $i\in\{1,\ldots,r\}$,
\item $m_{i+1}-m_i\leq \rho_2$, for all $i\in\{1,\ldots,r-1\}$. 
\end{enumerate}
\end{proposition}
\begin{proof}
  (a) 
  Suppose that there exists $i_0\in\{1,\ldots,r\}$ such that $m_{i_0}-\rho_{i_0} > m$. Let $D=\left\{m_{i_0}-\rho_j \mid j\in\{1,\ldots,{i_0}\}\right\}$. All the elements of $D$ are bigger than $m$, that is, $D\subseteq (m,\infty)$.
  On the other hand, by using Lemma~\ref{lema:divisors},  
  $D\subseteq \mathrm D(m_{i_0})$. Thus 
  $D\subseteq\mathrm D(m_{i_0})\cap (m,\infty)\subsetneqq \{m_1,\ldots,m_{i_0}\}$. The containment is strict since $m_1=m$. But this is absurd, since the two ends of the chain have the same cardinality.
  
  (b) Note that $m_{i+1} - \rho_2$ is a divisor of $m_{i+1}$. This implies that, if $m_{i+1} - \rho_2\ge m$, then  $m_{i+1} - \rho_2\in M$. As $m_{i+1} - \rho_2 < m_{i+1}$ and there is no element in $M$ strictly between $m_i$ and $m_{i+1}$, $m_{i+1}- \rho_2$ must be non greater than $m_i$.
\end{proof}
For efficiency reasons, the following result is important. It shows that we do not have to consider all divisors. 

\begin{proposition} \label{prop:amenable_def_generators}
A subset $M=\{m=m_1,\ldots,m_r\}$ of a numerical semigroup $S$  is $(S,m,r)$-amenable if and only if 
\begin{equation}\label{eq:def_amenable_generators}
\begin{array}{l}
 \mbox{ for all } i\in\{1,\ldots,r\} \mbox{ and } g \mbox{ minimal generator of } S,\\ 
 \mbox{ if } m_i-g\ge m, \mbox{ then } m_i-g\in  \{m_1,\ldots,m_r\}.
\end{array}
\end{equation}
\end{proposition}
\begin{proof}
 Let $m_i\in M$ and $u \in \mathrm D(m_i)\cap [m,\infty)$, with $u\ne m_i$. We shall prove that if (\ref{eq:def_amenable_generators}) holds, then $u\in M$, thus concluding that $M$ is $(S,m,r)$-amenable.
 We can write $u=m_i-\gamma$, with $\gamma \in S\setminus \{0\}$. 
 Assume as induction hypothesis that $m_i-\alpha\in\mathrm D(m_i)\cap [m,\infty)$ implies $m_i-\alpha\in M$, for all $\alpha$ less than $\gamma$. Let $g$ be a minimal generator that divides $\gamma$. As $\gamma-g<\gamma$, and $m_i-(\gamma-g)=m_i-\gamma+g\in \mathrm D(m_i)\cap [m,\infty)$, we have, by hypothesis, that $m_i-\gamma+g\in M$. But then, by (\ref{eq:def_amenable_generators}), $m_i-\gamma=(m_i-\gamma+g)-g\in M$.
\end{proof}

Propositions~\ref{prop:amenable_first_consequences} and \ref{prop:amenable_def_generators} led to an algorithm to compute the set of $(S,m,r)$-amenable sets. Pseudo-code is presented in Algorithm~\ref{alg:compute_amenables}~.

\begin{algorithm}\caption{$(S,m,r)$-amenable sets}
\label{alg:compute_amenables}
\DontPrintSemicolon
\SetKwInOut{Input}{Input}\SetKwInOut{Output}{Output}
\Input{A numerical semigroup $S$, $m\ge 2c-1$ and $r$ an integer}
\Output{The set of $(S,m,r)$-amenable sets}
\BlankLine

$SM := [[m]]$\tcc{the set of amenable sets}
Compute the generators $gens=\{n_1<\ldots < n_e\}$ and the elements $\{0=\rho_1<\rho_2<\ldots\}$ of $S$\;
\nl    \For{$i$ in $[2..r]$}{
        $newM := [\,]$\;
\nl        \For{$x$ in $SM$}{
            $min := Minimum(x[Length(x)]+\rho_2,m+\rho_i)$\tcc{the consequences in Proposition~\ref{prop:amenable_first_consequences} should be satisfied: the next element to be added must not be greater than the last + rho2 neither m+el[i]}
\nl\label{line:algo_amenable_for}            \For{$m_j$ in $[x[Length(x)]+1..min]$}{
\nl\label{line:algo_amenable_gens}                $divs := \{d\in m_j - gens\mid d>m\}$ \tcc{strict divisors of $m_j$ greater than m}
\nl\label{line:algo_amenable_if_condiii}           \If{$divs \subseteq x$}{\tcc{in order to get condition (\ref{eq:def_amenable_generators}) of Proposition~\ref{prop:amenable_def_generators} satisfied}
                            $Append(newM,[Union(x,[mj])])$}\;
            }   
        }
        $SM := newM$;
    }    

\return $SM$\;
\end{algorithm}

As we will see in the next subsection, we do not need all the amenable sets.

\subsection{The ground}\label{subsec:ground}
We continue considering $S$ a numerical semigroup minimally generated by $\{n_1<\cdots<n_e\}$ with conductor $c$. Let $m\geq 2c-1$. The set $\{m,\ldots, m+n_e-1\}$ is called the $(S,m)$-\emph{ground}, or simply \emph{ground}. 

The intersection of an $(S,m,r)$-amenable set $M$ with the $(S,m)$-ground is called the \emph{shadow} of $M$.

Note that the shadow of an amenable set is amenable.
\begin{lemma}\label{suelo-general}
Let $S$ be a numerical semigroup minimally generated by $\{n_1<\cdots<n_e\}$ with conductor $c$. Let $m\geq 2c-1$ and let $M=\{m=m_1<\cdots<m_r\}$ be an amenable set. Let $L=M\cap [m,m+n_e)$ be the shadow of $M$. Then
\[\mathrm D(M)=(M\setminus L)\cup \mathrm D(L),\]
and furthermore $\sharp\mathrm D(M)=\sharp(M\setminus L)+ \sharp\mathrm D(L)$.
\end{lemma}
\begin{proof}
The inclusion $(M\setminus L)\cup \mathrm D(L)\subseteq \mathrm D(M)$ is clear. For the other inclusion, let $x\in \mathrm D(M)\setminus (M\setminus L)=(\mathrm D(M)\setminus M)\cup L$. We want to prove that $x\in \mathrm D(L)$. Since $L\subseteq \mathrm D(L)$, we can assume that $x\in \mathrm D(M)\setminus M$. Then $x\in \mathrm D(m_i)$ for some $i\in \{1,\ldots,r\}$ and $m_i\geq m+n_e$. As $m_i-x\in S\setminus \{0\}$, there exists $j\in \{1,\ldots,e\}$ such that $m_i-x-n_j\in S$. Hence $x\in \mathrm D(m_i-n_j)$. By hypothesis $M$ is amenable
and thus $m_i-n_j\in M$, since $m_i-n_j\in \mathrm D(m_i)\cap[m,\infty)$. 
If needed, we can repeat the process until $m_i-n_j\in L$, that is, $x\in \mathrm D(L)$.

The second assertion follows easily since the above union is disjoint.
\end{proof}

As an easy but useful consequence, we get the following corollary. 
\begin{corollary}\label{cor:shadow_divisors}
Let $M$ and $N$ be $(m,r)$-amenable sets with shadows $L_M$ and $L_N$ respectively. $L_M\subseteq L_N \implies \sharp\mathrm D(M)\leq \sharp\mathrm D(N)$.
\end{corollary}
\begin{proof}
 Suppose that $L_N$ is the disjoint union of $L_M$ and a set $K$ of cardinality $k$. Observe that $\sharp (M\setminus L_M)=\sharp (N\setminus L_N)+k$.
 
 As $\mathrm D(L_N)=\mathrm D(L_M)\cup\mathrm D(K)\supseteq \mathrm D(L_M)\cup K$, it follows that $\sharp\mathrm D(L_N) \ge \sharp \mathrm D(L_M)+k$, that is, $\sharp\mathrm D(L_M) \le \sharp \mathrm D(L_N)-k$.
 
 $\sharp \mathrm D(M)= \sharp (M\setminus L_M)+ \sharp \mathrm D(L_M)\le \sharp (N\setminus L_N)+k+\sharp \mathrm D(L_N)-k$.
\end{proof}

\begin{corollary}\label{cor:delta-suelo}
Let $S$ be a numerical semigroup minimally generated by $\{n_1<\cdots<n_e\}$ with conductor $c$. Let $m\geq 2c-1$ and let $M\subset [m,\infty)$ be an amenable set which is an optimal configuration of cardinality $r$. Let $L=M\cap [m,m+n_e)$ be the shadow of $M$.
Then $\delta^r(m)=\sharp \mathrm D(L) + \sharp (M\setminus L)$.
\end{corollary}
\begin{corollary}\label{cor2:delta-suelo}
In particular, if there exists an optimal configuration $M$ of cardinality $r$ such that $ [m,m+n_e)\cap \mathbb N\subseteq M$, then $[ m, m+r-1+k]\cap\N$ is also an optimal configuration of cardinality $r+k$.
\end{corollary}

\subsection{An algorithm to compute generalized Feng-Rao numbers}\label{subsec:algorithm_FR}
In the cases where computing divisors is ``easy'', finding optimal configurations is as difficult as computing generalized Feng-Rao numbers. 
This problem is referred to as ``hard'' in the literature, even from the computational point of view. 

Algorithm \ref{alg:compute_FengRao} can be used to compute generalized Feng-Rao numbers of any numerical semigroup. Note that its efficiency depends on the number of amenable sets. 
Due to Corollary~\ref{cor:shadow_divisors}, it can be sharpened, since we only need to consider one amenable set for each possible shadow.

\begin{algorithm}[h]\caption{Generalized Feng-Rao numbers}
\label{alg:compute_FengRao}
\DontPrintSemicolon
\SetKwInOut{Input}{Input}\SetKwInOut{Output}{Output}
\Input{A numerical semigroup $S$, $m\in S$, $r\in \N$}
\Output{$\delta_{FR}^r(m)$}
\BlankLine
$SM := \emptyset$\;
\nl $AM := \{M\subset S\mid M \mbox{ is a } (S,m,r)\mbox{-amenable set}\}$\tcc{Compute the $(m,r)$-amenable sets, by making a call to Algorithm~\ref{alg:compute_amenables}}\;
\nl For each possible shadow $s$, add to $SM$ an element of $AM$ with shadow $s$, if it exists\;
$\nu:=m+r$\tcc{an obvious upper bound}
\nl \For{$M$ in $SM$}{
$D:=\bigcup\{Divisors(x)\mid x\in M\}$\tcc{Compute the divisors of $M$, by using Algorithm~\ref{alg:compute_divisors}}
$\nu := minimum (\sharp D,\nu)$}
\nl \return $\nu$
\end{algorithm}

This algorithm (even preliminary versions of it) has been extensively used by the authors to perform computations which gave the intuition that ultimately led to the main results of this paper.

\section{Numerical semigroups generated by intervals}

From now on we assume that $S=\langle a,\ldots, a+b\rangle$ with $a$ and $b$ positive integers, and $b<a$.

\subsection{Some counting lemmas}\label{subsec:counting_lemmas}
As we have seen above, it is crucial to know the number of divisors of subsets of the ground (this is obtained in Remark \ref{varias-expresiones}). In this section we prove some technical lemmas on counting the divisors of elements, and then apply them for elements in the ground. The main result (Lemma~\ref{seguidos-dan-menos}) shows that the minimum is obtained when the elements form an interval starting in $m$. 

Membership problem for semigroups generated by intervals is trivial as the following known result (and with many different formulations) shows.

\begin{lemma}\cite[Lemma 10, for $d=1$]{ChGsLl}\label{pertenencia-intervalos}
Let $k$ and $r$ be integers such that $0\leq r\leq a-1$. Then $ka+r\in S$ if and only if $r\le kb$.
\end{lemma}
\begin{lemma}\label{div-m+algo}
Let $m\geq 2c-1$. Let $q$ be a nonnegative integer and $j\in \{0,\ldots,a-1\}$.
\begin{multline*}
\mathrm D(m, m+qa+j)=\mathrm D(m) \\
\cup \left\{m-(ka+r)~|~ 0\leq r\leq a-j-1,\ \frac{r+j}b-q\leq k<\frac{r}b \right\}\\
\cup \left\{m-(ka+r)~|~ a-j\leq r\leq a-1,\ \frac{r+j-(a+b)}b-q\leq k<\frac{r}b \right\},
\end{multline*}
and this union is disjoint.
\end{lemma}
\begin{proof}
We describe the set $\mathrm D(m+qa+j)\setminus \mathrm D(m)$. Let $x$ be an integer such that $m+qa+j-x\in S$ and $m-x\not\in S$. In particular, as $m\geq 2c-1$, $m-x\not\in S$ implies that $m-x<c$, and thus $c-1\leq m-c<x$, which leads to $x\in S$. Thus $x\in \mathrm D(m+qa+j)\setminus \mathrm D(m)$. Set $n=m-x$, and let $k$ and $r$ be integers such that $n=ka+r$ ($x=m-(ka+r)$). Then $n=ka+r\not\in S$ and $n+qa+j=(q+k)a+(j+r)\in S$. In view Lemma \ref{pertenencia-intervalos}, this implies that $kb < r < a$ and
\begin{itemize}
\item if $r+j\leq a-1$, then $0\leq r+j\leq (q+k)b$,
\item if $r+j\geq a$, by writing $(q+k)a+(j+r)=(q+k+1)a+(j+r-a)$, we obtain $0\leq r+j-a\leq (q+k+1)b$.
\end{itemize}
\end{proof}

Figure \ref{fig:divisors_42_59} shows how are the divisors of $\mathrm D(m, m+\lambda)$ with $\lambda\in \{42,59\}$ and  $S=<9,10,11,12,13>$.

\begin{center}
\begin{figure}[h]
\begin{tikzpicture} 
[dm/.style={diamond,draw=black!50,fill=green!50,thick,
inner sep=0pt,minimum size=4.5000000000mm},
dv/.style={rectangle,draw=black!50,fill=red!50,thick,
inner sep=0pt,minimum size=4.5000000000mm},
dvcapdm/.style={diamond,draw=black!50,fill=green!50!red,thick,
inner sep=0pt,minimum size=4.5000000000mm},
dmcapground/.style={diamond,draw=black!50,fill=green!80,thick, 
inner sep=0pt,minimum size=4.5000000000mm}, 
dvcapground/.style={rectangle,draw=black!50,fill=red!80,thick, 
inner sep=0pt,minimum size=4.5000000000mm}, 
ground/.style={circle,draw=black!50,fill=blue!50,thick, 
inner sep=0pt,minimum size=4.5000000000mm},
gap/.style={circle,draw=black!40,fill=green!30,thick,
 inner sep=0pt,minimum size=4.5000000000mm}]
\node at (0,1) [dm, label = right: divisors of m] {};
\node at (3,1) [dv, label = right: divisors of v] {};
\node at (6,1) [dvcapdm, label = right: divisors of v $\cap$ divisors of m] {};
\node at (0,2) [dmcapground, label = right: divisors of m $\cap$ ground] {};
\node at (5,2) [dvcapground, label = right: divisors of v $\cap$ ground] {};
\node at (0,0) [ground, label = right: Other elements in ground] {};
\node at (6,0) [gap,label = right: Other elements] {};
\end{tikzpicture}
\begin{tikzpicture}
[dm/.style={diamond,draw=black!50,fill=green!50,thick,
inner sep=0pt,minimum size=4.5000000000mm},
dv/.style={rectangle,draw=black!50,fill=red!50,thick,
inner sep=0pt,minimum size=4.5000000000mm},
dvcapdm/.style={diamond,draw=black!50,fill=green!50!red,thick,
inner sep=0pt,minimum size=4.5000000000mm},
dmcapground/.style={diamond,draw=black!50,fill=green!80,thick, 
inner sep=0pt,minimum size=4.5000000000mm}, 
dvcapground/.style={rectangle,draw=black!80,fill=red!80,thick, 
inner sep=0pt,minimum size=4.5000000000mm}, 
ground/.style={circle,draw=black!50,fill=blue!50,thick, 
inner sep=0pt,minimum size=4.5000000000mm},
gap/.style={circle,draw=black!40,fill=green!30,thick,
 inner sep=0pt,minimum size=4.5000000000mm}]
\node at (0,2.4000000000) [ground] {\tiny 36};
\node at (0.5000000000,2.4666666666) [ground] {\tiny 37};
\node at (1,2.5333333333) [ground] {\tiny 38};
\node at (1.5000000000,2.6000000000) [ground] {\tiny 39};
\node at (2,2.6666666666) [ground] {\tiny 40};
\node at (2.5000000000,2.7333333333) [ground] {\tiny 41};
\node at (3,2.8000000000) [dvcapground] {\tiny 42};
\node at (3.5000000000,2.8666666666) [ground] {\tiny 43};
\node at (4,2.9333333333) [ground] {\tiny 44};
\node at (0,1.8000000000) [gap] {\tiny 27};
\node at (0.5000000000,1.8666666666) [gap] {\tiny 28};
\node at (1,1.9333333333) [dv] {\tiny 29};
\node at (1.5000000000,2) [dv] {\tiny 30};
\node at (2,2.0666666666) [dv] {\tiny 31};
\node at (2.5000000000,2.1333333333) [dv] {\tiny 32};
\node at (3,2.2000000000) [dv] {\tiny 33};
\node at (3.5000000000,2.2666666666) [gap] {\tiny 34};
\node at (4,2.3333333333) [dmcapground] {\tiny 35};
\node at (0,1.2000000000) [dv] {\tiny 18};
\node at (0.5000000000,1.2666666666) [dv] {\tiny 19};
\node at (1,1.3333333333) [dv] {\tiny 20};
\node at (1.5000000000,1.4000000000) [dv] {\tiny 21};
\node at (2,1.4666666666) [dvcapdm] {\tiny 22};
\node at (2.5000000000,1.5333333333) [dvcapdm] {\tiny 23};
\node at (3,1.6000000000) [dvcapdm] {\tiny 24};
\node at (3.5000000000,1.6666666666) [dm] {\tiny 25};
\node at (4,1.7333333333) [dm] {\tiny 26};
\node at (0,0.6000000000) [dvcapdm] {\tiny 9};
\node at (0.5000000000,0.6666666666) [dvcapdm] {\tiny 10};
\node at (1,0.7333333333) [dvcapdm] {\tiny 11};
\node at (1.5000000000,0.8000000000) [dvcapdm] {\tiny 12};
\node at (2,0.8666666666) [dvcapdm] {\tiny 13};
\node at (2.5000000000,0.9333333333) [gap] {\tiny 14};
\node at (3,1) [gap] {\tiny 15};
\node at (3.5000000000,1.0666666666) [gap] {\tiny 16};
\node at (4,1.1333333333) [gap] {\tiny 17};
\node at (0,0) [dvcapdm] {\tiny 0};
\node at (0.5000000000,0.0666666666) [gap] {\tiny 1};
\node at (1,0.1333333333) [gap] {\tiny 2};
\node at (1.5000000000,0.2000000000) [gap] {\tiny 3};
\node at (2,0.2666666666) [gap] {\tiny 4};
\node at (2.5000000000,0.3333333333) [gap] {\tiny 5};
\node at (3,0.4000000000) [gap] {\tiny 6};
\node at (3.5000000000,0.4666666666) [gap] {\tiny 7};
\node at (4,0.5333333333) [gap] {\tiny 8};
\end{tikzpicture}
\quad
\begin{tikzpicture}
[dm/.style={diamond,draw=black!50,fill=green!50,thick,
inner sep=0pt,minimum size=4.5000000000mm},
dv/.style={rectangle,draw=black!50,fill=red!50,thick,
inner sep=0pt,minimum size=4.5000000000mm},
dvcapdm/.style={diamond,draw=black!50,fill=green!50!red,thick,
inner sep=0pt,minimum size=4.5000000000mm},
dmcapground/.style={diamond,draw=black!50,fill=green!80,thick, 
inner sep=0pt,minimum size=4.5000000000mm}, 
dvcapground/.style={rectangle,draw=black!80,fill=red!80,thick, 
inner sep=0pt,minimum size=4.5000000000mm}, 
ground/.style={circle,draw=black!50,fill=blue!50,thick, 
inner sep=0pt,minimum size=4.5000000000mm},
gap/.style={circle,draw=black!40,fill=green!30,thick,
 inner sep=0pt,minimum size=4.5000000000mm}]
\node at (0,3.6000000000) [gap] {\tiny 54};
\node at (0.5000000000,3.6666666666) [gap] {\tiny 55};
\node at (1,3.7333333333) [gap] {\tiny 56};
\node at (1.5000000000,3.8000000000) [gap] {\tiny 57};
\node at (2,3.8666666666) [gap] {\tiny 58};
\node at (2.5000000000,3.9333333333) [dv] {\tiny 59};
\node at (3,4) [gap] {\tiny 60};
\node at (3.5000000000,4.0666666666) [gap] {\tiny 61};
\node at (4,4.1333333333) [gap] {\tiny 62};
\node at (0,3) [ground] {\tiny 45};
\node at (0.5000000000,3.0666666666) [dvcapground] {\tiny 46};
\node at (1,3.1333333333) [dvcapground] {\tiny 47};
\node at (1.5000000000,3.2000000000) [dvcapground] {\tiny 48};
\node at (2,3.2666666666) [dv] {\tiny 49};
\node at (2.5000000000,3.3333333333) [dv] {\tiny 50};
\node at (3,3.4000000000) [gap] {\tiny 51};
\node at (3.5000000000,3.4666666666) [gap] {\tiny 52};
\node at (4,3.5333333333) [gap] {\tiny 53};
\node at (0,2.4000000000) [dvcapground] {\tiny 36};
\node at (0.5000000000,2.4666666666) [dvcapground] {\tiny 37};
\node at (1,2.5333333333) [dvcapground] {\tiny 38};
\node at (1.5000000000,2.6000000000) [dvcapground] {\tiny 39};
\node at (2,2.6666666666) [dvcapground] {\tiny 40};
\node at (2.5000000000,2.7333333333) [dvcapground] {\tiny 41};
\node at (3,2.8000000000) [ground] {\tiny 42};
\node at (3.5000000000,2.8666666666) [ground] {\tiny 43};
\node at (4,2.9333333333) [ground] {\tiny 44};
\node at (0,1.8000000000) [dv] {\tiny 27};
\node at (0.5000000000,1.8666666666) [dv] {\tiny 28};
\node at (1,1.9333333333) [dv] {\tiny 29};
\node at (1.5000000000,2) [dv] {\tiny 30};
\node at (2,2.0666666666) [dv] {\tiny 31};
\node at (2.5000000000,2.1333333333) [dv] {\tiny 32};
\node at (3,2.2000000000) [dv] {\tiny 33};
\node at (3.5000000000,2.2666666666) [dv] {\tiny 34};
\node at (4,2.3333333333) [dvcapdm] {\tiny 35};
\node at (0,1.2000000000) [dv] {\tiny 18};
\node at (0.5000000000,1.2666666666) [dv] {\tiny 19};
\node at (1,1.3333333333) [dv] {\tiny 20};
\node at (1.5000000000,1.4000000000) [dv] {\tiny 21};
\node at (2,1.4666666666) [dvcapdm] {\tiny 22};
\node at (2.5000000000,1.5333333333) [dvcapdm] {\tiny 23};
\node at (3,1.6000000000) [dvcapdm] {\tiny 24};
\node at (3.5000000000,1.6666666666) [dvcapdm] {\tiny 25};
\node at (4,1.7333333333) [dvcapdm] {\tiny 26};
\node at (0,0.6000000000) [dvcapdm] {\tiny 9};
\node at (0.5000000000,0.6666666666) [dvcapdm] {\tiny 10};
\node at (1,0.7333333333) [dvcapdm] {\tiny 11};
\node at (1.5000000000,0.8000000000) [dvcapdm] {\tiny 12};
\node at (2,0.8666666666) [dvcapdm] {\tiny 13};
\node at (2.5000000000,0.9333333333) [gap] {\tiny 14};
\node at (3,1) [gap] {\tiny 15};
\node at (3.5000000000,1.0666666666) [gap] {\tiny 16};
\node at (4,1.1333333333) [gap] {\tiny 17};
\node at (0,0) [dvcapdm] {\tiny 0};
\node at (0.5000000000,0.0666666666) [gap] {\tiny 1};
\node at (1,0.1333333333) [gap] {\tiny 2};
\node at (1.5000000000,0.2000000000) [gap] {\tiny 3};
\node at (2,0.2666666666) [gap] {\tiny 4};
\node at (2.5000000000,0.3333333333) [gap] {\tiny 5};
\node at (3,0.4000000000) [gap] {\tiny 6};
\node at (3.5000000000,0.4666666666) [gap] {\tiny 7};
\node at (4,0.5333333333) [gap] {\tiny 8};
\end{tikzpicture}
\caption{$\mathrm D(m,m+\lambda)$}\label{fig:divisors_42_59}
\end{figure}
\end{center}

\begin{remark}\label{suelo-intervalos}
For the particular case $q=0$ and $0<j<a$, we get
\[ \mathrm D(m,m+j)=\mathrm D(m)\cup \{m-(ka+r)~|~ a-j\leq r\leq a-1,\ 0<r-kb \leq (a+b)-j\}.\]

For $q=1$ and $j=0$,
\[ \mathrm D(m,m+a)=\mathrm D(m)\cup \{m-(ka+r) ~|~ 0\leq r \le a-1,\ 0<r-kb\le b\},\]
which is the same as above by taking $j=a$.

For the case $q=1$, we get
\begin{multline*}
\mathrm D(m, m+a+j)=\mathrm D(m) \\
\cup \left\{m-(ka+r)~|~ 0\leq r\leq a-j-1,\ 0<r-kb \le b-j \right\}\\
\cup \left\{m-(ka+r)~|~ a-j\leq r\leq a-1,\ 0<r-kb \le (a+b)+b-j\right\}.
\end{multline*}

This describes all elements $\mathrm D(m,m+\ell)$, with $\ell\in \{1,\ldots,a+b-1\}$ (i.e., $m+\ell$ in the ground).
\end{remark}

\begin{lemma}
\label{d-L}
Let $0=i_0<i_1<\cdots<i_t<i_{t+1}< a+b$ be such that $\{m,m+i_1,\ldots,m+i_{t+1}\}$ is amenable.
Then
\begin{multline*}
\mathrm D(m,m+i_1,\ldots,m+i_{t+1})=\mathrm D(m,m+i_1,\ldots, m+i_t)\cup\\ \{m-(ka+r)~|~ a-i_{t+1}\leq r\leq a-1-i_t,\ 0< r-kb\leq (a+b)-i_{t+1} \}.\end{multline*}

\end{lemma}
\begin{proof}
Assume first that $i_{t+1}\le a$. Note that $\mathrm D(m+i_{t+1})\setminus(\mathrm D(m,m+i_1,\ldots,m+i_t))=\bigcap_{j=0}^t \mathrm D(m+i_{t+1})\setminus \mathrm D(m+i_j)$, and this equals
\[
\bigcap_{j=0}^t \{m+i_j-(ka+r') ~|~ a-(i_{t+1}-i_j)\leq r'\leq a-1,\ 0< r'-kb \le (a+b)-(i_{t+1}-i_j)\}
\]
(Remark \ref{suelo-intervalos}). If we make the change of variables $r=r'-i_j$ for each $j$,
we obtain
\[
\bigcap_{j=0}^t \{m-(ka+r) ~|~ a-i_{t+1}\leq r\leq a-1-i_j,\ -i_j< r-kb \le (a+b)-i_{t+1}\}.
\]
Intersecting means choosing the least intervals for $r$ and $r-kb$, and we get the desired result.

Now assume that $a<i_{t+1}< a+b$. By hypothesis there exists $s$ such that $i_{t+1}-i_s< a$ and $i_{t+1}-i_{s-1}\ge a$ (by amenability). 
For $i_{t+1}-i_j\ge a$, write $i_{t+1}-i_j=a+h_j$. Hence $\bigcap_{j=0}^t \mathrm D(m+i_{t+1})\setminus \mathrm D(m+i_j)$ equals
\begin{multline*}
\bigcap_{j=s}^t \{m+i_j-(ka+r') ~|~ a-(i_{t+1}-i_j)\leq r'\leq a-1,\ 0< r'-kb \le (a+b)-(i_{t+1}-i_j)\}\\
\bigcap \Big( \bigcap_{j=0}^{s-1} \big( \{m+i_j-(ka+r') ~|~ 0\leq r'\leq a-(i_{t+1}-i_j-a)-1,\ 0< r'-kb \le b-(i_{t+1}-i_j-a)\} \\
\cup \{m+i_j-(ka+r') ~|~ a-(i_{t+1}-i_j-a)\leq r'\leq a-1,\ 0< r'-kb \le (a+b)+b-(i_{t+1}-i_j-a)\} \big) \Big).
\end{multline*}
If we perform again the change of variables $r=r'-i_j$, we obtain that $C=\bigcap_{j=s}^t \{m+i_j-(ka+r') ~|~ a-(i_{t+1}-i_j)\leq r'\leq a-1,\ 0< r'-kb \le (a+b)-(i_{t+1}-i_j)\} = \{m-(ka+r) ~|~ a-i_{t+1}\leq r\leq a-i_t-1,\ -i_s< r-kb \le (a+b)-i_{t+1}\}$. Analogously, for every $j\in \{0,\ldots,s-1\}$,
\begin{multline*}
\{m+i_j-(ka+r') ~|~ 0\leq r'\leq a-(i_{t+1}-i_j-a)-1,\ 0< r'-kb \le b-(i_{t+1}-i_j-a)\} \\
\cup \{m+i_j-(ka+r') ~|~ a-(i_{t+1}-i_j-a)\leq r'\leq a-1,\ 0< r'-kb \le (a+b)+b-(i_{t+1}-i_j-a)\}
\end{multline*}
equals
\begin{multline*}
\{m-(ka+r) ~|~ -i_j\leq r\leq 2a-i_{t+1}-1,\ -i_j< r-kb \le a+b-i_{t+1}\} \\
\cup \{m-(ka+r) ~|~ 2a-i_{t+1}\leq r\leq a-i_j-1,\ -i_j< r-kb \le 2(a+b)-i_{t+1}\}.
\end{multline*}
Observe that $a-i_t-1\le 2a-i_{t+1}-1$ if and only if $i_{t+1}-i_t\le a$, which is the case since we are using an amenable set.
Hence $C$ does not cut the second set in the above union, and the whole intersection is as in the case $i_{t+1}<a$.
\end{proof}

\begin{corollary}\label{cuantos-mas}
Let $0=i_0<i_1<\cdots<i_t<i_{t+1}< a+b$. Then
\[\sharp\mathrm D(m,m+i_1,\ldots,m+i_{t+1})=\sharp\mathrm D(m,m+i_1,\ldots, m+i_t)+  \sum_{j=i_t+1}^{i_{t+1}} \left\lceil \frac{a+b-j}b\right\rceil - \left\lceil \frac{i_{t+1}-j}b\right\rceil.\]
\end{corollary}
\begin{proof}
We compute the cardinality of $\{m-(ka+r)~|~ a-i_{t+1}\leq r\leq a-1-i_t,\ 0< r-kb\leq (a+b)-i_{t+1} \}$. Note that $m-(ka+r)=m-(k'a+r')$ with $0\leq r,r'<a$ implies that $k=k'$ and $r=r'$. Thus we must calculate $\sharp \{(k,r)~|~ a-i_{t+1}\leq r\leq a-1-i_t,\ 0< r-kb\leq (a+b)-i_{t+1} \}$, which equals $\sharp\{(k,r)~|~ i_t+1\leq a-r\leq i_{t+1},\ \frac{i_{t+1}+r-(a+b)}b\leq k < \frac{r}b\}$. By taking $j=a-r$, we get \[\sum_{j=i_{t}+1}^{i_{t+1}}\sharp \Big\{k~|~ \frac{i_{t+1}-j}b-1\leq k < \frac{a-j}b\Big\}= \sum_{j=i_t+1}^{i_{t+1}} \Big(\big(\Big\lceil \frac{a-j}b\Big\rceil-1\big)-\big(\Big\lceil \frac{i_{t+1}-j}b\Big\rceil-1\big)+1\Big),\]
and the proof follows easily.
\end{proof}

\begin{remark}\label{varias-expresiones}
 Recall that \[\sharp\mathrm D(m,m+i_1,\ldots,m+i_{t})=\sharp\mathrm D(m,m+i_1,\ldots, m+i_{t-1})+ \sum_{j=i_{t-1}+1}^{i_{t}} \left\lceil \frac{a+b-j}b\right\rceil - \left\lceil \frac{i_{t}-j}b\right\rceil,\]
and by applying several times this process we obtain
\begin{multline*}
\sharp\mathrm D(m,m+i_1,\ldots,m+i_{t})=\sharp\mathrm D(m)+ \sum_{k=1}^t \sum_{j=i_{k-1}+1}^{i_{k}} \left\lceil \frac{a+b-j}b\right\rceil - \left\lceil \frac{i_{k}-j}b\right\rceil  \\= \sharp \mathrm D(m)+ \sum_{j=1}^{i_t} \pe[\frac{a+b-j}b] - \sum_{k=1}^t \sum_{j=i_{k-1}+1}^{i_{k}} \left\lceil \frac{i_{k}-j}b\right\rceil.
\end{multline*}

And thus
\[\sharp\mathrm D(m,m+i_1,\ldots,m+i_{t})=\sharp \mathrm D(m)+ \sum_{j=1}^{i_t} \pe[\frac{a+b-j}b] - \sum_{k=1}^t \sum_{j=1}^{i_{k}-i_{k-1}} \left\lceil \frac{(i_{k}-i_{k-1})-j}b\right\rceil.\]
If we write $d_k=i_k-i_{k-1}$, this rewrites as
\[\sharp\mathrm D(m,m+i_1,\ldots,m+i_{t})=\sharp \mathrm D(m)+ \sum_{j=1}^{i_t} \pe[\frac{a+b-j}b] - \sum_{k=1}^t \sum_{j=1}^{d_k} \left\lceil \frac{d_k-j}b\right\rceil.\]
Hence the value of $\sharp\mathrm D(m,m+i_1,\ldots,m+i_t)$ depends on $d_1,\ldots,d_t$, subject to $\sum_{k=1}^t d_k=i_t$.
\end{remark}

\begin{corollary}\label{cuantos-mas-seguidos}
Let $t\in\{1,\ldots,a+b-1\}$. Then
\[\sharp\mathrm D(m,m+1,\ldots,m+t)=\sharp \mathrm D(m)+\sum_{j=1}^t \left\lceil \frac{a+b-j}b\right\rceil.\]
\end{corollary}

As a consequence of this, when the shadow of a configuration is an interval containing $m$, then we can compute the number of divisors of its elements. 

\begin{corollary}\label{cor:divs_intervals}
 Let $m$ be an integer greater than or equal to $2c-1$. Let $m=m_1<m_2<\cdots<m_t$ be integers such that $\{m_1,m_2,\ldots,m_t\}$ is amenable.
Assume that $l=\sharp\{m_1,\ldots,m_t\}\cap [m,m+a+b)=\{m,m+1,\ldots,m+l-1\}$. Then
\[ \sharp \mathrm D(m_1,\ldots,m_t)= m-2g+t+\sum_{j=1}^{l-1}\pe[\frac{a-j}b].\]
\end{corollary}
\begin{proof}
  This is a direct consequence of Proposition \ref{prop:well-known}, Lemma \ref{suelo-general}, and Corollary \ref{cuantos-mas-seguidos}.
\end{proof}
Indeed, we will show that among these configurations there is an optimal one. To do this, we first prove that the best shadows are those of the form $\{m,m+1,\ldots,m+l-1\}$, and later (in the next section) we will have to compute the smallest possible value of $l$.

Let us see how to compute sums of the form $\sum_{j=1}^t \left\lceil \frac{a-j}b\right\rceil$. 

\begin{lemma}\label{sumatoria-pe}
Let $x$ and $y$ be postive integers. Assume that $y=cb+r$ with $c$ an integer and $0\leq r<b$, and that $k$ is an integer such that $kb \le x-r <(k+1)b$. Then
\begin{enumerate}
\item if $r\neq 0$, $\sum_{j=1}^x \pe[\frac{y-j}b] = (c+1)(r-1)+b\sum_{i=0}^{k-1} (c-i)+(x-(kb+r)+1)(c-k)=x(c-k)+(k+1)(r-1)+b\frac{k(k+1)}2$,
\item if $r=0$, $\sum_{j=1}^x \pe[\frac{y-j}b] = -c+b\sum_{i=0}^{k-1} (c-i)+(x-kb+1)(c-k)=
(x+1)(c-k)+b\frac{k(k+1)}2-c$.
\end{enumerate}
\end{lemma}
\begin{proof}
 Observe that
   \[\sum_{j=1}^x \pe[\frac{y-j}b] = \sum_{j=1}^{r-1} \pe[\frac{y-j}b] + \sum_{j=r}^{r+b-1}\pe[\frac{y-j}b]+\cdots + \sum_{j=(k-1)b+r}^{kb+r-1}\pe[\frac{y-j}b] + \sum_{j=kb+r}^x\pe[\frac{y-j}b],\]
and
  \[\sum_{j=r+lb}^{(l+1)b+r-1} \pe[\frac{y-j}b] = \sum_{j=0}^{b-1}\pe[\frac{y-(r+lb)-j}b] = \sum_{j=0}^{b-1}\pe[\frac{(c-l)b-j}b] = b(c-l).\]
In the same way the first and last summand are computed. If $r=0$, the first summand does not appear, and the second sum starts on $0$, and so we have to decrease the total amount by $\pe[\frac{cb}b]=c$.
\end{proof}

Actually, as we see next it suffices to consider the following type of sums.

\begin{remark}
 For the case $x=y$, we get $k=c$ and
\begin{enumerate}
\item if $r\neq 0$, $\sum_{j=1}^x \pe[\frac{x-j}b] = (c+1)(r-1)+b\frac{c(c+1)}2=\frac{c+1}2(x+r)-c-1$,
\item if $r=0$, $\sum_{j=1}^x \pe[\frac{x-j}b] = \frac{c+1}2x-c$.
\end{enumerate}
Observe also that $\sum_{j=1}^x \pe[\frac{x-j}b]=\sum_{j=1}^{x-1} \pe[\frac{j}b]$.
\end{remark}

The following trick will allow us to prove that the best possible shadows are those that are intervals starting in $m$.

\begin{remark} 
Let $d_k=i_k-i_{k-1}$, $i\in \{1,\ldots, t\}$. Then $\sum_{k=1}^t d_k=i_t$. If we replace $\{d_1,d_2\}$ with $\{1,d_1+d_2-1\}$, the total sum of the $d_k$'s remains the same (we are thus assuming that both $d_1$ and $d_2$ are greater than one). Let us see what happens to
\[\sum_{k=1}^t \sum_{j=1}^{i_{k}-i_{k-1}} \left\lceil \frac{(i_{k}-i_{k-1})-j}b\right\rceil = \sum_{k=1}^t \sum_{j=1}^{d_k} \pe[\frac{d_k-j}b].\]

Write $d_k=c_k b+r_k$, with $c_k$ and integer and $0\leq r_k<b$. Set  $s_k=\sum_{j=1}^{d_k} \pe[\frac{d_k-j}b]$. Then $s_k= \frac{c_k+1}2(d_k+r_k)-1-c_k$, if $r_k\neq 0$, and $s_k=\frac{c_k+1}2d_k-c_k$, otherwise. Let $c$ and $r$ be the quotient and remainder of the division of $d_1+d_2-1$ by $b$. Let $\Delta=\sum_{j=1}^{d_1+d_2-1} \pe[\frac{d_1+d_2-1-j}b] -s_1- s_2$. If $b=1$, then $r_1=r_2=r=0$, $c_i=d_i$, $c=d_1+d_2-1$, and $\Delta=d_1d_2$, which is a nonnegative integer. For $b>1$ we distinguish three cases depending on the value of $r_1+r_2$.

\begin{itemize}
 \item If $r_1+r_2=0$ (this means $r_1=r_2=0$), then $d_1=c_1b$, $d_2=c_2b$, $c=c_1+c_2-1$, and $r=b-1$. Then $\Delta= b c_1 c_2- (c_1+c_2)$. Since we are assuming that $b\geq 2$, and $c_1$ and $c_2$ are positive integers, this amount is nonnegative.

\item If $0<r_1+r_2\leq b$, then $c=c_1+c_2$ and $r=r_1+r_2-1$. Thus if $rr_1r_2\neq0$, $\Delta=bc_1c_2+c_1(r_2-1)+c_2(r_1-1)$, which is greater than or equal to zero. For $r=0$, either $r_1=0$ (and $r_2=1$) or $r_2=0$ (and $r_1=1$). Assume without loss of generality that $r_1=0$. We obtain $\Delta=(bc_1-1)c_2$, which is again nonnegative.

\item Finally if $r_1+r_2\geq b+1$, then $c=c_1+c_2+1$ and $r=r_1+r_2-b-1$. In this setting $r_1\neq0\neq r_2$. If $r\neq0$, then $\Delta = bc_1c_2+c_1(r_2-1)+c_2(r_1-1)+r_1+r_2-(b+2)$, which is nonnegative since $r_1+r_2\ge b+2$. For $r_1+r_2=b+1$ ($r=0$), we obtain a nonnegative $\Delta=bc_1c_2+c_1(r_2-1)+c_2(r_1-1)$.
\end{itemize}
\end{remark}

With all this in mind, we are able to prove the main result of this section, that is, if the elements in the ground form an interval containing $m$, then we get the least possible number of divisors.

\begin{lemma} Let $0=i_0<i_1<\cdots<i_t< a+b$. Then
\label{seguidos-dan-menos}
 \[\sharp\mathrm D(m,m+i_1,\ldots,m+i_t)\ge \sharp \mathrm D(m,m+1,\ldots,m+t). \]
\end{lemma}
\begin{proof}
 Let $d_k=i_k-i_{k-1}$ for $k\in \{1,\ldots,t\}$ ($i_0=0$). We know that (see Remark \ref{varias-expresiones})
  \[\sharp\mathrm D(m,m+i_1,\ldots,m+i_t) = \sharp \mathrm D(m) + \sum_{j=1}^{i_t} \pe[\frac{a+b-j}b] - \sum_{k=1}^t \sum_{j=1}^{d_k} \pe[\frac{d_k-j}b].\]
 By applying several times the above remark, we obtain that $\sharp \mathrm D(m,m+i_1,\ldots,m+i_t)\geq \sharp \mathrm D(m,m+i_1',m+i_2',\ldots,m+i_t')$, with $i_t=i_t'$ and $d_k'=i_k'-i_{k-1}'=1$ for $k\in \{2,\ldots,t\}$. By using again the above expression, but now for $d_k'$ instead of $d_k$, we get
  \[\sharp\mathrm D(m,m+i_1',\ldots,m+i_t') = \sharp \mathrm D(m) + \sum_{j=1}^{i_t} \pe[\frac{a+b-j}b] - \sum_{j=1}^{i_1'} \pe[\frac{i_1'-j}b].\]
  Hence by Corollary \ref{cuantos-mas-seguidos}, in order to prove the inequality of the statement, it suffices to show that
  \[\sum_{j=1}^{i_t}\pe[\frac{a+b-j}b]-\sum_{j=1}^{i_1'} \pe[\frac{i_1'-j}b] \ge \sum_{j=1}^t \pe[\frac{a+b-j}b],\]
  or equivalently,
  \[\sum_{j=t+1}^{i_t} \pe[\frac{a+b-j}b] - \sum_{j=1}^{i_1'}\pe[\frac{i_1'-j}b]\ge 0.\]
  Now, if we take into account that $\sum_{j=t+1}^{i_t} \pe[\frac{a+b-j}b]= \sum_{j=1}^{i_t-t} \pe[\frac{a+b-t-j}b]$, that $\sum_{j=1}^{i_1'}\pe[\frac{i_1'-j}b]=\sum_{j=1}^{i_1'-1}\pe[\frac{i_1'-j}b]$, and that $i_1'+(t-1)=i_t'=i_t$, we get $\sum_{j=t+1}^{i_t} \pe[\frac{a+b-j}b]= \sum_{j=1}^{i_1'-1} \pe[\frac{a+b-t-j}b]$. Since $a+b-t\ge i_t+1-t=i_1'$, we obtain the desired inequality.
\end{proof}


\subsection{Ordered amenable sets}\label{subsec:ordered_amenable}

As we have seen above, the minimum number of divisors of elements in the ground is reached when these elements form an interval starting in $m$. 
In this section we study configurations fulfilling this condition.

Let $M$ be a configuration and let
$$j_0=max\{j\in\{0,\ldots, a-1\}\mid x-(m+b)= qa +j, \mbox{for some } q\in \Z \mbox{ and } x\in M\}.$$ 
Let us call \emph{wagon} of $M$ the set  $\{x\in M\mid x-(m+b)= qa +j_0, \mbox{for some integer } q\}$.

An element $P$ of a configuration $M$ is said to be the \emph{pivot} of $M$ if either
($P<m+b$ and $P$ is the maximum of $M$) or,  $P$ is the maximum of the wagon.

Note that the wagon so as the pivot element of a configuration can be determined in an algorithmic way. 

Checking Figure~\ref{fig:biggest_ordered_amenable} may be useful. The wagon is column 22 and the pivot is 305. 

The wagon consists of the rightmost elements of $M$. The highest of these is the pivot element. Note that the index of the column containing the wagon is $b+j_0$. 

An $(S,m,r)$-amenable set $M$ with pivot element $P$ is said to be \emph{ordered amenable} if its shadow is of the form $\{m,m+1,\ldots,m+t\}$, for some integer $t$, $0\le t<a+b-1$, and the only element that can possibly be added to obtain an $(S,m,r+1)$-amenable set without increasing the shadow is $P+\rho_2$.

\begin{figure}[h]
\begin{tikzpicture}
[trg/.style={diamond,draw=black!80,fill=red!50,thick,
inner sep=0pt,minimum size=4.5000000000mm},
rect/.style={rectangle,draw=black!70,fill=green!40,thick,
inner sep=0pt,minimum size=4.5000000000mm},
wagon/.style={rectangle,draw=black!85,fill=blue!40,thick,
inner sep=0pt,minimum size=4.5000000000mm},
pivot/.style={star,draw=black!80,fill=blue!60,thick,
inner sep=0pt,minimum size=4.5000000000mm},
ground/.style={circle,draw=black!50,fill=blue!30,thick, 
inner sep=0pt,minimum size=4.5000000000mm},
int_triang/.style={diamond,draw=black!80,fill=red!80,thick, 
inner sep=0pt,minimum size=4.5000000000mm}, 
int_rect/.style={rectangle,draw=black!80,fill=black!10!green!90,thick, 
inner sep=0pt,minimum size=4.5000000000mm}, 
int_wagon/.style={rectangle,draw=black!85,fill=blue!80,thick, 
inner sep=0pt,minimum size=4.5000000000mm}, 
gap/.style={circle,draw=black!40,fill=green!20,thick,
 inner sep=0pt,minimum size=4.5000000000mm}]
\node at (2.2000000000,3.3868421052) [gap] {{\color{gray}{\tiny 288}}};
\node at (2.7500000000,3.4210526315) [gap] {{\color{gray}{\tiny 289}}};
\node at (3.3000000000,3.4552631578) [gap] {{\color{gray}{\tiny 290}}};
\node at (3.8500000000,3.4894736842) [gap] {{\color{gray}{\tiny 291}}};
\node at (4.4000000000,3.5236842105) [gap] {{\color{gray}{\tiny 292}}};
\node at (4.9500000000,3.5578947368) [gap] {{\color{gray}{\tiny 293}}};
\node at (5.5000000000,3.5921052631) [gap] {{\color{gray}{\tiny 294}}};
\node at (6.0500000000,3.6263157894) [gap] {{\color{gray}{\tiny 295}}};
\node at (6.6000000000,3.6605263157) [gap] {{\color{gray}{\tiny 296}}};
\node at (7.1500000000,3.6947368421) [gap] {{\color{gray}{\tiny 297}}};
\node at (7.7000000000,3.7289473684) [gap] {{\color{gray}{\tiny 298}}};
\node at (8.2500000000,3.7631578947) [gap] {{\color{gray}{\tiny 299}}};
\node at (8.8000000000,3.7973684210) [gap] {{\color{gray}{\tiny 300}}};
\node at (9.3500000000,3.8315789473) [gap] {{\color{gray}{\tiny 301}}};
\node at (9.9000000000,3.8657894736) [gap] {{\color{gray}{\tiny 302}}};
\node at (10.4500000000,3.9000000000) [gap] {{\color{gray}{\tiny 303}}};
\node at (11,3.9342105263) [rect] {\tiny 304};
\node at (11.5500000000,3.9684210526) [pivot] {\tiny 305};
\node at (12.1000000000,4.0026315789) [gap] {{\color{gray}{\tiny 306}}};
\node at (2.2000000000,2.7368421052) [gap] {{\color{gray}{\tiny 269}}};
\node at (2.7500000000,2.7710526315) [gap] {{\color{gray}{\tiny 270}}};
\node at (3.3000000000,2.8052631578) [gap] {{\color{gray}{\tiny 271}}};
\node at (3.8500000000,2.8394736842) [gap] {{\color{gray}{\tiny 272}}};
\node at (4.4000000000,2.8736842105) [gap] {{\color{gray}{\tiny 273}}};
\node at (4.9500000000,2.9078947368) [gap] {{\color{gray}{\tiny 274}}};
\node at (5.5000000000,2.9421052631) [gap] {{\color{gray}{\tiny 275}}};
\node at (6.0500000000,2.9763157894) [gap] {{\color{gray}{\tiny 276}}};
\node at (6.6000000000,3.0105263157) [gap] {{\color{gray}{\tiny 277}}};
\node at (7.1500000000,3.0447368421) [gap] {{\color{gray}{\tiny 278}}};
\node at (7.7000000000,3.0789473684) [gap] {{\color{gray}{\tiny 279}}};
\node at (8.2500000000,3.1131578947) [gap] {{\color{gray}{\tiny 280}}};
\node at (8.8000000000,3.1473684210) [trg] {\tiny 281};
\node at (9.3500000000,3.1815789473) [trg] {\tiny 282};
\node at (9.9000000000,3.2157894736) [trg] {\tiny 283};
\node at (10.4500000000,3.2500000000) [trg] {\tiny 284};
\node at (11,3.2842105263) [rect] {\tiny 285};
\node at (11.5500000000,3.3184210526) [wagon] {\tiny 286};
\node at (12.1000000000,3.3526315789) [gap] {{\color{gray}{\tiny 287}}};
\node at (2.2000000000,2.0868421052) [gap] {{\color{gray}{\tiny 250}}};
\node at (2.7500000000,2.1210526315) [gap] {{\color{gray}{\tiny 251}}};
\node at (3.3000000000,2.1552631578) [gap] {{\color{gray}{\tiny 252}}};
\node at (3.8500000000,2.1894736842) [gap] {{\color{gray}{\tiny 253}}};
\node at (4.4000000000,2.2236842105) [gap] {{\color{gray}{\tiny 254}}};
\node at (4.9500000000,2.2578947368) [gap] {{\color{gray}{\tiny 255}}};
\node at (5.5000000000,2.2921052631) [gap] {{\color{gray}{\tiny 256}}};
\node at (6.0500000000,2.3263157894) [gap] {{\color{gray}{\tiny 257}}};
\node at (6.6000000000,2.3605263157) [trg] {\tiny 258};
\node at (7.1500000000,2.3947368421) [trg] {\tiny 259};
\node at (7.7000000000,2.4289473684) [trg] {\tiny 260};
\node at (8.2500000000,2.4631578947) [trg] {\tiny 261};
\node at (8.8000000000,2.4973684210) [trg] {\tiny 262};
\node at (9.3500000000,2.5315789473) [trg] {\tiny 263};
\node at (9.9000000000,2.5657894736) [trg] {\tiny 264};
\node at (10.4500000000,2.6000000000) [trg] {\tiny 265};
\node at (11,2.6342105263) [rect] {\tiny 266};
\node at (11.5500000000,2.6684210526) [wagon] {\tiny 267};
\node at (12.1000000000,2.7026315789) [gap] {{\color{gray}{\tiny 268}}};
\node at (2.2000000000,1.4368421052) [gap] {{\color{gray}{\tiny 231}}};
\node at (2.7500000000,1.4710526315) [gap] {{\color{gray}{\tiny 232}}};
\node at (3.3000000000,1.5052631578) [gap] {{\color{gray}{\tiny 233}}};
\node at (3.8500000000,1.5394736842) [gap] {{\color{gray}{\tiny 234}}};
\node at (4.4000000000,1.5736842105) [trg] {\tiny 235};
\node at (4.9500000000,1.6078947368) [trg] {\tiny 236};
\node at (5.5000000000,1.6421052631) [trg] {\tiny 237};
\node at (6.0500000000,1.6763157894) [trg] {\tiny 238};
\node at (6.6000000000,1.7105263157) [trg] {\tiny 239};
\node at (7.1500000000,1.7447368421) [trg] {\tiny 240};
\node at (7.7000000000,1.7789473684) [trg] {\tiny 241};
\node at (8.2500000000,1.8131578947) [trg] {\tiny 242};
\node at (8.8000000000,1.8473684210) [trg] {\tiny 243};
\node at (9.3500000000,1.8815789473) [trg] {\tiny 244};
\node at (9.9000000000,1.9157894736) [trg] {\tiny 245};
\node at (10.4500000000,1.9500000000) [trg] {\tiny 246};
\node at (11,1.9842105263) [rect] {\tiny 247};
\node at (11.5500000000,2.0184210526) [wagon] {\tiny 248};
\node at (12.1000000000,2.0526315789) [gap] {{\color{gray}{\tiny 249}}};
\node at (2.2000000000,0.7868421052) [trg] {\tiny 212};
\node at (2.7500000000,0.8210526315) [trg] {\tiny 213};
\node at (3.3000000000,0.8552631578) [trg] {\tiny 214};
\node at (3.8500000000,0.8894736842) [trg] {\tiny 215};
\node at (4.4000000000,0.9236842105) [trg] {\tiny 216};
\node at (4.9500000000,0.9578947368) [trg] {\tiny 217};
\node at (5.5000000000,0.9921052631) [trg] {\tiny 218};
\node at (6.0500000000,1.0263157894) [trg] {\tiny 219};
\node at (6.6000000000,1.0605263157) [trg] {\tiny 220};
\node at (7.1500000000,1.0947368421) [trg] {\tiny 221};
\node at (7.7000000000,1.1289473684) [trg] {\tiny 222};
\node at (8.2500000000,1.1631578947) [trg] {\tiny 223};
\node at (8.8000000000,1.1973684210) [trg] {\tiny 224};
\node at (9.3500000000,1.2315789473) [trg] {\tiny 225};
\node at (9.9000000000,1.2657894736) [trg] {\tiny 226};
\node at (10.4500000000,1.3000000000) [trg] {\tiny 227};
\node at (11,1.3342105263) [rect] {\tiny 228};
\node at (11.5500000000,1.3684210526) [wagon] {\tiny 229};
\node at (12.1000000000,1.4026315789) [gap] {{\color{gray}{\tiny 230}}};
\node at (0,0) [int_triang] {\tiny 189};
\node at (0,0) [below=5pt] {{\color{blue}{\tiny 1}}};
\node at (0.5500000000,0.0342105263) [int_triang] {\tiny 190};
\node at (0.5500000000,0.0342105263) [below=5pt] {{\color{blue}{\tiny 2}}};
\node at (1.1000000000,0.0684210526) [int_triang] {\tiny 191};
\node at (1.1000000000,0.0684210526) [below=5pt] {{\color{blue}{\tiny 3}}};
\node at (1.6500000000,0.1026315789) [int_triang] {\tiny 192};
\node at (1.6500000000,0.1026315789) [below=5pt] {{\color{blue}{\tiny 4}}};
\node at (2.2000000000,0.1368421052) [int_triang] {\tiny 193};
\node at (2.2000000000,0.1368421052) [below=5pt] {{\color{blue}{\tiny 5}}};
\node at (2.7500000000,0.1710526315) [int_triang] {\tiny 194};
\node at (2.7500000000,0.1710526315) [below=5pt] {{\color{blue}{\tiny 6}}};
\node at (3.3000000000,0.2052631578) [int_triang] {\tiny 195};
\node at (3.3000000000,0.2052631578) [below=5pt] {{\color{blue}{\tiny 7}}};
\node at (3.8500000000,0.2394736842) [int_triang] {\tiny 196};
\node at (3.8500000000,0.2394736842) [below=5pt] {{\color{blue}{\tiny 8}}};
\node at (4.4000000000,0.2736842105) [int_triang] {\tiny 197};
\node at (4.4000000000,0.2736842105) [below=5pt] {{\color{blue}{\tiny 9}}};
\node at (4.9500000000,0.3078947368) [int_triang] {\tiny 198};
\node at (4.9500000000,0.3078947368) [below=5pt] {{\color{blue}{\tiny 10}}};
\node at (5.5000000000,0.3421052631) [int_triang] {\tiny 199};
\node at (5.5000000000,0.3421052631) [below=5pt] {{\color{blue}{\tiny 11}}};
\node at (6.0500000000,0.3763157894) [int_triang] {\tiny 200};
\node at (6.0500000000,0.3763157894) [below=5pt] {{\color{blue}{\tiny 12}}};
\node at (6.6000000000,0.4105263157) [int_triang] {\tiny 201};
\node at (6.6000000000,0.4105263157) [below=5pt] {{\color{blue}{\tiny 13}}};
\node at (7.1500000000,0.4447368421) [int_triang] {\tiny 202};
\node at (7.1500000000,0.4447368421) [below=5pt] {{\color{blue}{\tiny 14}}};
\node at (7.7000000000,0.4789473684) [int_triang] {\tiny 203};
\node at (7.7000000000,0.4789473684) [below=5pt] {{\color{blue}{\tiny 15}}};
\node at (8.2500000000,0.5131578947) [int_triang] {\tiny 204};
\node at (8.2500000000,0.5131578947) [below=5pt] {{\color{blue}{\tiny 16}}};
\node at (8.8000000000,0.5473684210) [int_triang] {\tiny 205};
\node at (8.8000000000,0.5473684210) [below=5pt] {{\color{blue}{\tiny 17}}};
\node at (9.3500000000,0.5815789473) [int_triang] {\tiny 206};
\node at (9.3500000000,0.5815789473) [below=5pt] {{\color{blue}{\tiny 18}}};
\node at (9.9000000000,0.6157894736) [int_triang] {\tiny 207};
\node at (9.9000000000,0.6157894736) [below=5pt] {{\color{blue}{\tiny 19}}};
\node at (10.4500000000,0.6500000000) [int_triang] {\tiny 208};
\node at (10.4500000000,0.6500000000) [below=5pt] {{\color{blue}{\tiny 20}}};
\node at (11,0.6842105263) [int_rect] {\tiny 209};
\node at (11,0.6842105263) [below=5pt] {{\color{blue}{\tiny 21}}};
\node at (11.5500000000,0.7184210526) [int_wagon] {\tiny 210};
\node at (11.5500000000,0.7184210526) [below=5pt] {{\color{blue}{\tiny 22}}};
\node at (12.1000000000,0.7526315789) [ground] {{\color{gray}{\tiny 211}}};
\node at (12.1000000000,0.7526315789) [below=5pt] {{\color{blue}{\tiny 23}}};
\end{tikzpicture}
\caption{The biggest ordered amenable.}\label{fig:biggest_ordered_amenable}
\end{figure}
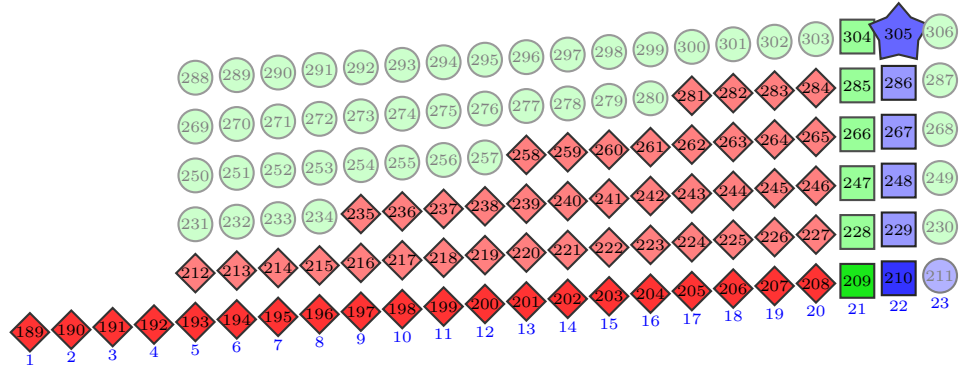

We now show how to construct $(S,m,r)$-amenable sets.

\begin{remark}
 In view of Lemma \ref{div-m+algo}, for every positive integer $q$ with $qb<a$,
\[ \mathrm D(m+qa+qb)=\mathrm D(m)\cup \{m-(ka+r)~|~ a-qb\le r\le a-1, \frac{r-(a+b)}b\le k< \frac{r}b\}.\]
By performing a change of variables (change $a-r$ to $r$ and $-k-1$ to $k$), we obtain that 
\[\mathrm D(m+qa+qb)=\mathrm D(m) \cup \{m+ka+r~|~ 1\le r\le qb, -\frac{a-r}b-1< k\le \frac{r}b\}.\]
Hence 
\[\mathrm D(m+qa+qb)\cap [m,\infty)=\{m\}\cup\{m+ka+r~|~ 1\le r\le qb, 0\le k\le \frac{r}b\}\]
and 
\[\mathrm D(m+qa+qb)\cap [m,m+a+b)=\{m,m+1,\ldots,m+qb\}\]
(observe that for $k$ to be one, $r$ must be at least $b$, and in this case we obtain $m+a+b$ which is not in $[m,m+a+b)$).

Moreover, 
\[\sharp \mathrm D(m+qa+qb)\cap [m,\infty)=1+q+\frac{b}2q(q+1),\]
since the cardinality of the set $\{m+ka+r~|~ 1\le r\le qb, 0\le k\le \frac{r}b\}$ is $\sum_{r=1}^{qb} (\lfloor \frac{r}b\rfloor +1)$, which can be rewritten as $qb+\sum_{i=0}^{q-1} \sum_{j=0}^{b-1}\lfloor \frac{ib+j}b\rfloor + \lfloor\frac{qb}b\rfloor$, and this equals $qb+q+\sum_{i=0}^{q-1} bi= qb+q +b\frac{q(q-1)}2= q+\frac{b}2q(q+1)$. By adding now the cardinality of $\{m\}$, we obtain the desired equality. 
\end{remark}

Clearly the sets $\mathrm D(m+\lambda)\cap [m,\infty)$ are amenable sets. 
The following lemma shows that some of these are indeed  ordered amenable sets. The problem is that their cardinalities do not cover all possible $r$'s.

\begin{lemma}
 Let $q$ be a positive integer such that $qb<a$. Then $\mathrm D(m+qa+qb)\cap [m,\infty)$ is an ordered amenable set.
\end{lemma}
\begin{proof}
 We already know that its shadow is an interval containing $m$ (the condition $qb<a$, ensures that the shadow is not the whole ground), and as pointed out above, it is amenable. In order to conclude the proof, we show that for all $s>m$, $s\not \in \mathrm D(m+qa+qb)$, if the set $\mathrm D(m+qa+qb)\cup\{s\}$ is amenable, then its shadow is larger than $\{m,m+1,\ldots,m+qb\}$. Write $s=m+ua+v$, with $0\le v<a$. If $u$ is zero, as $s\not\in \mathrm D(m+qa+qb)$, we obtain that $v>qb$, obtaining in this way a new element in the shadow. So $u$ must be positive. We distinguish to cases.
\begin{itemize}
 \item If $v\le qb$, then as $s\not \in \mathrm D(m+qa+qb)$, by the preceding remark, we deduce that $u$ must be greater than $\frac{v}b$. But then $m+ua+v-(m+a+b-1)=(u-1)a+(v-b)+1$, and this element is in $S$ if and only if $v-b+1\le (u-1)b$ (Lemma \ref{pertenencia-intervalos}), or equivalently, $v< ub$, which holds since $u>\frac{v}b$. This proves that $m+a+b-1$ is in the shadow of $\mathrm D(m+qa+qb)\cup \{s\}$ (under the assumption that this set is amenable), and it is not in $\{m,m+1,\ldots,m+qb\}$, a contradiction.

 \item Now assume that $v>qb$. Then the element $m+v$ is in the shadow of $\mathrm D(m+qa+qb)\cup\{s\}$, obtaining again a contradiction.
\end{itemize}

\end{proof}

From an $(S,m,r)$-ordered amenable set, we can construct another $(S,m,r-1)$-ordered amenable set, just by removing its pivot.

\begin{lemma} \label{quitando-pivot}
If $M$ is an ordered amenable set, which shadow is not the whole ground, and $P$ is its pivot, then $M\setminus\{P\}$ is ordered amenable.
\end{lemma}
\begin{proof}
Observe that $P$ does not belong to $\mathrm D(M\setminus\{P\})$, and thus $M\setminus\{P\}$ is still amenable. From the definition of pivot, it follows easily that this set is also ordered amenable.
\end{proof}

Let $r$ be a positive integer, there exists $q\in \mathbb Z$ such that
\[ q+\frac{1}2bq(q-1)\le r< 1+q+\frac{1}2bq(q+1).\]
Define $\mathrm h(r)=q$. Thus we can write $r=\mathrm h(r)+ \frac{1}2 b\mathrm h(r)(\mathrm h(r)-1)+s$, with $0\le s\le \mathrm h(r)b$. Hence 
\begin{equation}\label{r-hr}
r=\mathrm h(r)+ \frac{1}2 b\mathrm h(r)(\mathrm h(r)-1)+k\mathrm h(r)+j,
\end{equation} with $-1\le k\le b-1$ and $0<j\le \mathrm h(r)$ ($k=-1$ only in the case $r=\mathrm h(r)+ \frac{1}2 b\mathrm h(r)(\mathrm h(r)-1)$, and then $j=\mathrm h(r)$).
Note that $\mathrm h(r)=0$ leads to $r=0$, so we may assume that $\mathrm h(r)>0$. Observe also that $j+k=0$ only when $\mathrm h(r)=1=j$ and $k=-1$.

\begin{proposition} \label{existencia-r-ordered}
Let $r$ be a positive integer. Let $k$ and $j$ be as above. If $b(\mathrm h(r)-1)+k+1<a+b-1$, then the set 
\begin{multline*} (\mathrm D(m+(\mathrm h(r)-1)(a+b))\cap [m,\infty))\\ \cup \{m+ua+v~|~ (\mathrm h(r)-1)b+1\le v\le (\mathrm h(r)-1)b+k, 0\le u\le \mathrm h(r)-1\}\\  \cup  \{m+((\mathrm h(r)-1)b+k+1)a+v~|~ 0\le v<j\}
\end{multline*}
is an $r$-ordered amenable set.
\end{proposition}
\begin{proof}
 This set is obtained from $\mathrm D(m+\mathrm h(r)(a+b))$ by repeating the Lemma \ref{quitando-pivot}  $1+\mathrm h(r)+\frac{b}2\mathrm h(r)(\mathrm h(r)+1)-r$ times.
\end{proof}

Next we prove that ordered $(S,m,r)$-amenable sets have minimal shadow in the set of all $(S,m,r)$-amenable sets with shadow an interval containing $m$. As a consequence any two $(S,m,r)$-ordered amenable sets have the same shadow.

\begin{proposition}\label{prop:ordered_shadow}
Let $M$ be an ordered $(S,m,r)$-amenable subset of $S$ whose shadow has $t$ elements and let $N$ be another $(S,m,r)$-amenable subset of $S$ whose shadow is an interval containing $m$. Then, the shadow of  $N$ has at least also $t$ elements. 
\end{proposition}
\begin{proof}
 Suppose that the shadow of $N$ has less than $t$ elements. This implies that the set $N\setminus M$ is non empty (since both sets have cardinality $r$) and therefore it has a minimum $z$. Furthermore, $N$ has no elements in the wagon of $M$.
 It is straightforward to observe that $M\cup\{z\}$ is amenable.\\ 
  In fact, as $N$ is amenable, we have that $\mathrm D(z)\cap[m,\infty)\subset N$; $(\mathrm D(z)\setminus \{z\})\cap[m,\infty)\subset M$ because $z$ is minimum. So $\mathrm D(z)\cap[m,\infty)\subset M\cup\{z\}$.
 As $z$ is not in the column containing the wagon of $M$ we conclude that $z\ne P+\rho_2$, which contradicts the assumption that $M$ is ordered.
\end{proof}

Considering $M$ and $N$ ordered amenable sets in the above proposition and applying it in both directions, we get the following consequence.

\begin{corollary}\label{cor:ordered_shadow}
 The shadows of ordered $(S,m,r)$-amenable sets coincide.
\end{corollary}

With all these ingredients we can effectively compute the cardinality of the shadow of an ordered $(S,m,r)$-amenable set.

\begin{corollary}\label{suelo-r-ordered}
 Let $M$ be an ordered $(S,m,r)$-amenable set, and let $k$ and $j$ be as in (\ref{r-hr}), then $\# (M\cap [m,m+a+b))= (\mathrm h(r)-1)b+k+2$.
\end{corollary}
\begin{proof}
 As any two ordered amenable sets with the same cardinality have the same elements in the ground, we can use the ordered ameneable set of the preceding proposition. Observe that the ground for this set is $\{m,m+1,\ldots,m+(\mathrm h(r)-1)b,m+(\mathrm h(r)-1)b+1,\ldots,m+(\mathrm h(r)-1)+k+1\}$.
\end{proof}

Observe that this result gives a bound for integers $r$ such that there exists an ordered $(S,m,r)$-amenable set.

Now we prove that if $M$ is an $(S,m,r)$-amenable set whose shadow is not an interval containing $m$, then we can remove the trailing spaces in the shadow without increasing the number of divisors, that is we can find $N$, an $(S,m,r)$-amenable with shadow  an interval containing $m$ and such that $\sharp \mathrm D(N)\le \sharp \mathrm D(M)$. By using what we already know for ordered $(S,m,r)$-amenable set, as a consequence we will obtain that they are optimal configurations. To this end we need several tools.

The first one enables us to push an $(S,m,r)$-amenable set to the right, obtaining an $(S,m+1,r)$-amenable set.
\begin{lemma}\label{lemma:push_right}
Let $M$ be an $(S,m,r)$-amenable set. Then $N=M+1=\{x+1\mid x\in M\}$ is an $(S,m+1,r)$-amenable set.
\end{lemma}
\begin{proof}
Let $x=y+1\in N$, with $y\in M$, and suppose that $h\in S$ is such that the divisor $x-h$ of $x$ is greater than $m$. We have to prove that $x-h\in N$. As $y-h$ is greater than or equal to $m$, we have that $y-h\in M$. It follows that $x-h=(y+1)-h=(y-h)+1\in M+1=N$.
\end{proof}
If we shift an $(S,m,r)$-amenable set to the left, we get an $(S,m-1,r)$-amenable set (provided $m-1\ge 2c-1$).

\begin{lemma}\label{lemma:push_left}
Let $M$ be an $(S,m,r)$-amenable set with $m\ge 2c$. Then $M-1=\{x-1\mid x\in M\}$ is an $(S,m-1,r)$-amenable set.
\end{lemma}
\begin{proof}
Let $y\in M-1$, say $y=x-1$, with $x\in M$. Note that $y\ge m$. Now we use Corollary~\ref{cor:big_divisors}. By hypothesis $\mathrm D(x)\cap [m,\infty)= (x-S) \cap [m,\infty)\subseteq M$, but then 
$\mathrm D(y)\cap [m,\infty)= ((x-1)-S) \cap [m,\infty)\subseteq (M-1)$.
\end{proof}

If we add the element $m$ to an $(S,m+1,r)$-amenable set, we get an $(S,m,r+1)$-amenable set.

\begin{lemma}\label{lemma:add_element}
Let $N$ be an $(S,m+1,r)$-amenable set. The set $M=\{m\}\cup N$ is $(S,m,r+1)$-amenable.
\end{lemma}
\begin{proof}
If $N=\{m+1=m_2<\ldots <m_{r+1}\}$ then $M=\{m<m_2<\ldots <m_{r+1}\}$. Write $m_1=m$. 

We have to check that $M$ is $m$-closed under division. It clearly holds for $i=1$. Let $i\ge 2$. The divisors of $m_i$ non smaller than $m+1$ belong to $N$ and thus to $M$. Therefore, divisors of $m_i$ non smaller than $m$ belong to $M$.
\end{proof}
It is immediate that if we remove the biggest element of an $(m,r)$-amenable set, then we get an $(m,r-1)$-amenable set (provided $r> 1$). 
\begin{lemma}\label{lemma:remove_maximum}
Let $M$ be an $(m,r)$-amenable set and suppose that $r>1$. Let $u$ be the maximum of $M$. Then $M\setminus\{u\}$ is an $(m,r-1)$-amenable set.
\end{lemma}

Before removing the trailing spaces of an $(S,m,r)$-amenable set, we need it to no contain the last element in the ground, that is $m+a+b-1$. If this is the case, next we give a procedure to obtain another $(S,m,r)$-amenable set whose shadow is at most as large as the original set, but not containing $m+a+b-1$.

\begin{proposition}\label{prop:remove_upper_element_in_ground}
 Given an $(S,m,r)$-amenable set $M$ with shadow $L_M$ not coinciding with the ground, we can construct an $(S,m,r)$-amenable set $N$ with shadow $L_N$ not containing $m+a+b-1$ and such that $\sharp L_N \le \sharp L_M$.
\end{proposition}
\begin{proof}
 Assume that $M$ is an $(S,m,r)$-amenable set containing $m+a+b-1$, and with a shadow different to the ground. Then, there is at least an element $x$ in the ground, such that $x$ is not in $M$, and thus $x<m+a+b-1$. Let $N=\{m\}\cup(M+1)\setminus\{\max\{\{m\}\cup(M+1)\}\}$. Thus $N$ is a shifting to the right, and then its maximum is replaced by $m$. So $N$ is by the preceding lemmas an $(S,m,r)$-amenable set. Observe also that $x+1\not\in N$. Hence we repeat this procedure until $x+k$ becomes $m+a+b-1$.
\end{proof}

Suppose we have a configuration not containing $m+a+b-1$. We can shrink it so that the shadow of the configuration obtained is an interval containing $m$. It can be done using the following results. 

\begin{lemma}\label{partition}
 Let $M$ be an $(S,m,r)$-amenable set not containing $m+a+b-1$. Assume there exists a column $c$, such that $c\cap M=\emptyset$, and if $M_1$ are the elements in $M$ in the columns to the left of $c$, and $M_2=M\setminus M_1$, then $M_2\neq \emptyset$ ($c$ is a splitting column for $M$). Let $r_1=\sharp M_1$, $r_2=\sharp M_2$ and $m_2=\min M_2$. Then
\begin{enumerate}
 \item $M_1$ is an $(S,m,r_1)$ amenable set,
 \item $M_2$ is an $(S,m_2,r_2)$ amenable set,
 \item $M_1\cup(M_2-1)$ is an $(S,m,r)$ amenable set.
\end{enumerate}  
\end{lemma}
\begin{proof}
 It suffices to show that no element in $M_1$ divides an element in $M_2$, and vice-versa. Assume that there is $x\in M_1$ and $y\in M_2$ such that $y-x\in S$, that is, $y-x=ka+r$ for some $r,k$ nonnegative integers with $r\le \min\{a-1,kb\}$ (Lemma \ref{pertenencia-intervalos}). Hence $y=x+ka+r$, and $y-(ka+i)\in \mathrm D(M)\cap [m,\infty)=M$ for all $i\in \{0,\ldots,r\}$. Assume that $c$ corresponds with the elements $s$ in $[m,\infty)\cap \mathbb N$ such that $s-(m+b)\bmod a=j$. Then, by hypothesis 
 \begin{multline*}
 y-r-(m+b)\bmod a = y-(ka+r)-(m+b)\bmod a\\
 = x-(m+b)\bmod a<y-(m+b)\bmod a,  
 \end{multline*}
 and thus there is $i\in \{0,\ldots,r\}$ such that $y-(ka+i)-(m+b)\bmod a=y-i-(m+b)\bmod a=j$, contradicting that $c$ was an empty column of $M$.

 Now assume that $x\in M_1$ and $y\in M_2$ are such that $x-y=ka+r$ for some $r,k$ as above. In this setting, $x-(ka+i)\in M$ for all $i\in \{0,\ldots,r\}$. By hypothesis $x-(m+b)\bmod a< y-(m+b)\bmod a$. And \begin{multline*}
 y+r-(m+b)\bmod a =y+ka+r-(m+b)\bmod a\\=x-(m+b)\bmod a< y-(m+b)\bmod a.                                                                                                                                                                                                              \end{multline*}
 It follows that for some $i\in\{0,\ldots,r\}$, $y+i-(m+b)\bmod a=a-1$, but this is impossible, since as $m+a+b-1\not\in M$, the column $\{s\in M ~|~ s-(m+b)\bmod a=a-1\}$ is empty. 
\end{proof}

\begin{proposition}\label{prop:shadow_interval}
Let $M$ be an $(S,m,r)$-amenable set whose shadow $L_M$ has $t$ elements. There exists an $(S,m,r)$-amenable set $T$ whose shadow $L_T$ is an interval containing $m$ and has no more than $t$ elements, i.e., $\sharp L_T \le \sharp L_M$. 
\end{proposition}
\begin{proof}
 Every time you find a splitting column as in the statement of Lemma \ref{partition}, change $M$ with $M_1\cup(M_2-1)$. This procedure does not increase the number of elements in the shadow of $M$.
\end{proof}

\begin{lemma}\label{amenable-a-amenable-intervalo}
Let $M$ be an $(S,m,r)$-amenable set whose shadow $L_M$ has $t$ elements. There exists an $(S,m,r)$-amenable set $T$ whose shadow is an interval containing $m$, and $\sharp \mathrm D(T) \le \sharp \mathrm D(M)$. 
\end{lemma}
\begin{proof}
Let $T$ be as in Proposition \ref{prop:shadow_interval}, and assume that $\sharp L_T=t-k$. In view of Lemma \ref{seguidos-dan-menos}, $\sharp \mathrm D(L_M)\ge \sharp \mathrm D(m,m+1,\ldots,m+t-1)$. By Corollary \ref{cuantos-mas-seguidos}, $\sharp\mathrm D(m,m+1,\ldots,m+t-1)=\sharp \mathrm D(m)+\sum_{j=1}^{t-1} \left\lceil \frac{a+b-j}b\right\rceil$, and as $\left\lceil \frac{a+b-j}b\right\rceil\ge 1$, this amount is greater than or equal to $\sharp \mathrm D(m)+\sum_{j=1}^{t-k-1} \left\lceil \frac{a+b-j}b\right\rceil+k$, which according to Corollary \ref{cuantos-mas-seguidos} equals $\sharp \mathrm D(m,m+1,\ldots,m+t-k-1)+k=\sharp\mathrm D(L_T)+k$. Now we use Lemma \ref{suelo-general}, having $\sharp \mathrm D(M)=\sharp \mathrm D(L_M)+\sharp M\setminus L_M=\sharp \mathrm D(L_M)+r-t\ge \sharp \mathrm D(L_T)+k+r-t=\sharp \mathrm D(L_T)+\sharp T\setminus L_T=\sharp \mathrm D(T)$.
\end{proof}

\begin{theorem}\label{ordered-son-optimas}
 Let $S$ be a numerical semigroup with conductor $c$, and let $m\geq 2c-1$. Then every ordered $(S,m,r)$-amenable set is an optimal configuration.
\end{theorem}
\begin{proof}
 Let $M$ be an ordered $(S,m,r)$-amenable set. By Proposition \ref{condiciones-ms}, among the optimal configurations, there is always an $(S,m,r)$-amenable set. Let $N$ be an $(S,m,r)$-amenable set that is an optimal configuration. In light of Lemma \ref{amenable-a-amenable-intervalo}, we can assume that its shadow is an interval containing $m$. By Proposition \ref{prop:ordered_shadow}, the shadow of $M$ is contained in that of $N$, and by Corollary \ref{cor:shadow_divisors}, we get that $\#\mathrm D(M)\le \#\mathrm D(N)$. As $N$ is an optimal configuration we deduce that $\#\mathrm D(M)=\#\mathrm D(N)$, and thus $M$ is also an optimal configuration.
\end{proof}

\begin{corollary}
 Let $S=\langle a,a+1,\ldots,a+b\rangle$ with integers $a,b$ such that $0<b<a$. 
 Write $r$ as in formula (\ref{r-hr}), that is, $r=\mathrm h(r)+ \frac{1}2 b\mathrm h(r)(\mathrm h(r)-1)+k\mathrm h(r)+j$, with $-1\le k\le b-1$ and $0<j\le \mathrm h(r)$.
 Then $\mathrm E(r,\langle a,a+1,\ldots,a+b\rangle)$ equals
\[  r-1 +\left\{
  \begin{array}{ll}
    \sum_{i=1}^{b(\mathrm h(r)-1)+k+1} \pe[\frac{a-i}b], & \hbox{if } b(\mathrm h(r)-1)+k+2<a+b, \\
    \sum_{i=1}^{a+b-1} \pe[\frac{a-i}b], & \hbox{otherwise}.
  \end{array}
\right.
\]
\end{corollary}
\begin{proof}
Assume that $b(\mathrm h(r)-1)+k+1<a+b-1$. Then by Proposition \ref{existencia-r-ordered}, there exists an ordered $(S,m,r)$-amenable set. In this setting the proof follows from Corollaries \ref{cor:divs_intervals} and \ref{suelo-r-ordered} and Theorem \ref{ordered-son-optimas}.

Observe also that if $b(\mathrm h(r)-1)+k+1=a+b-2$, by Proposition \ref{existencia-r-ordered}, the set 
\begin{multline*} M=(\mathrm D(m+(\mathrm h(r)-1)(a+b))\cap [m,\infty))\\ \cup \{m+ua+v~|~ (\mathrm h(r)-1)b+1\le v\le (\mathrm h(r)-1)b+k+1, 0\le u\le \mathrm h(r)-1\}
\end{multline*} 
is an ordered amenable set, and thus by Theorem \ref{ordered-son-optimas} an optimal configuration for $r=\#M$. As $\mathrm D(M\cup\{m+a+b-1\})=\mathrm D(M)\cup\{m+a+b-1\}$, the set $M\cup\{m+a+b-1\}$ is an optimal configuration of cardinality $r+1$, whose shadow fills the whole ground. By using now Corollary \ref{cor2:delta-suelo}, we get optimal configurations for cardinalies greater than $r+1$. And the proof follows easily by Corollary \ref{cor:divs_intervals}.
\end{proof}


Needless to say that, by using this formula for numerical semigroups generated by intervals, 
we have no need of the general Algorithm \ref{alg:compute_FengRao}, speeding-up the computation 
of Feng-Rao distances for such semigroups. 

\begin{remark}

The reader can check that we have ${\rm E}(r,S)=\rho_{r}$ exactly in the following cases: 

\begin{enumerate}

\item[(A)] If either $r=b\sigma(p)+1,\ldots,b\sigma(p)+p+1$ and the ground is not completely filled, or 

\item[(B)] $r\geq r_{0}$, where $r_{0}$ is the first $r$ filling the ground, 

\end{enumerate}

being $\sigma(p):=1+\cdots+p=\frac{1}{2}p(p+1)$ and $p\geq 1$. 

Besides, since both sequences ${\rm E}(r,S)$ and $\rho_r$ are 
strictly increasing, the largest difference between them is for $r=2$ 
and for the first $r$ after each element in the first case, that is,
$r=b\sigma(p)+p+2$ where $\rho_{r}$ jumps from one interval to the next.

\end{remark}



\begin{thebibliography}{100}

\bibitem{AngCar} A. Barbero and C. Munuera, 
\lq\lq The weight hierarchy of Hermitian codes\rq\rq, 
{\em SIAM J. Discrete Math.}\/ vol. 13, no. 1, pp.79-104 (2000). 

\bibitem{WSPink} A. Campillo and J.I. Farr\'{a}n, 
\lq\lq Computing Weierstrass semigroups and the Feng-Rao distance from singular plane models\rq\rq, 
{\em Finite Fields and their Applications}\/ {\bf 6}, pp. 71-92 (2000). 

\bibitem{Arf} A. Campillo, J.I. Farr\'{a}n and C. Munuera, 
\lq\lq On the parameters of algebraic geometry codes related to Arf semigroups\rq\rq, 
{\em IEEE Trans. of Information Theory}\/ {\bf 46}, pp. 2634-2638 (2000). 

\bibitem{ChGsLl} S. T. Chapman, P. A. Garc\'{\i}a-S\'{a}nchez and D. Llena, 
\lq\lq The catenary and tame degree of a numerical semigroup\rq\rq, 
{\em Forum Math.}\/ {\bf 21}, pp. 117-129 (2009). 

\bibitem{numericalsgps} M. Delgado, P. A. Garc\'{\i}a-S\'{a}nchez and J. Morais, 
\lq\lq NumericalSgps\rq\rq, {\em A GAP package for numerical semigroups}\/, 
current version number 0.96 (2008). 
Available via {\tt http://www.gap-system.org/}. 

\bibitem{JMDA} J. I. Farr\'{a}n, P. A. Garc\'{\i}a-S\'{a}nchez and D. Llena, 
{\em On the Feng-Rao numbers}\/, 
Actas de las VII Jornadas de Matem\'{a}tica Discreta y Algor\'{\i}tmica, pp. 321-333 (2010). 

\bibitem{WCC} J. I. Farr\'{a}n and C. Munuera, 
\lq\lq Goppa-like bounds for the generalized Feng-Rao distances\rq\rq, 
{\em Discrete Applied Mathematics}\/ {\bf 128}/1, pp. 145-156 (2003). 

\bibitem{FR} G.L. Feng and T.R.N. Rao, 
\lq\lq Decoding algebraic-geometric codes up to the designed minimum distance\rq\rq, 
{\em IEEE Trans. Inform. Theo\-ry}\/ {\bf 39}, pp. 37-45 (1993). 

\bibitem{HeiPel} P. Heijnen and R. Pellikaan, 
\lq\lq Generalized Hamming weights of $q$\/-ary Reed-Muller codes\rq\rq, 
{\em IEEE Trans. Inform. Theo\-ry}\/ {\bf 44}, pp. 181-197 (1998). 

\bibitem{HKM} T. Helleseth, T. Kl\o ve and J. Mykkleveit, 
\lq\lq The weight distribution of irreducible cyclic codes 
with block lengths $n_{1}((q^{l}-1)/N)$\/\rq\rq, 
{\em Discrete Math.}\/, vol. 18, pp. 179-211 (1977). 

\bibitem{HvLP} T. H\o holdt, J.H. van Lint and R. Pellikaan, 
\lq\lq Algebraic Geometry codes\rq\rq, 
in {\em Handbook of Coding Theory}\/, 
V. Pless, W.C. Huffman and R.A. Brualdi, Eds., 
pp. 871-961 (vol. 1), Elsevier, Amsterdam (1998). 

\bibitem{KirPel} C. Kirfel and R. Pellikaan, 
\lq\lq The minimum distance of codes in an array coming from telescopic semigroups\rq\rq, 
{\em IEEE Trans. Inform. Theory}\/ {\bf 41}, pp. 1720-1732 (1995). 

\bibitem{NS} J. C. Rosales and P. A. Garc\'{\i}a-S\'{a}nchez, 
\lq\lq Numerical Semigroups\rq\rq, 
{\em Developments in Maths.}\/ vol. 20, Springer (2010). 

\bibitem{Wei} V. Wei, 
\lq\lq Generalized Hamming weights for linear codes\rq\rq, 
{\em IEEE Trans. Inform. Theory}\/ {\bf 37}, pp. 1412-1428 (1991). 

\end{thebibliography}
\end{document}